\documentclass[a4paper,reqno,11pt]{amsart}

\usepackage{amsmath}
\usepackage{amsthm}
\usepackage{amssymb}
\usepackage{latexsym}
\usepackage[dvipdfmx]{graphicx}
\usepackage{booktabs}
\usepackage[dvipdfm,bookmarkstype=toc=true,colorlinks=true,urlcolor=black,linkcolor=black,citecolor=black,linktocpage=true,bookmarks=false]{hyperref}

\numberwithin{equation}{section}

\newtheorem{thm}{Theorem}[section]
\newtheorem{prop}[thm]{Proposition}
\newtheorem{lem}[thm]{Lemma}
\newtheorem{cor}[thm]{Corollary}

\theoremstyle{definition}
\newtheorem{defn}[thm]{Definition}

\theoremstyle{remark}
\newtheorem{rem}[thm]{Remark}

\allowdisplaybreaks[3]

\newcommand{\al}{\alpha}
\newcommand{\be}{\beta}
\newcommand{\ga}{\gamma}

\newcommand{\de}{\delta}

\newcommand{\e}{\varepsilon}
\newcommand{\fy}{\varphi}
\newcommand{\la}{\lambda}

\newcommand{\sgm}{\sigma}
\newcommand{\ze}{\zeta}
\renewcommand{\th}{\theta}
\newcommand{\p}{\partial}
\newcommand{\med}{\mathrm{med}}
\newcommand{\I}{\infty}
\newcommand{\Sc}[1]{\mathcal{#1}}
\newcommand{\F}{\Sc{F}}

\newcommand{\FR}[1]{\mathfrak{#1}}
\newcommand{\Bo}[1]{\mathbb{#1}}
\newcommand{\R}{\Bo{R}}
\newcommand{\T}{\Bo{T}}
\newcommand{\Tg}{\Bo{T}_\ga}
\newcommand{\Zg}{\Bo{Z}_\ga}
\newcommand{\lec}{\lesssim}
\newcommand{\gec}{\gtrsim}
\newcommand{\hhat}{\widehat}
\newcommand{\bbar}{\overline}
\newcommand{\ti}{\widetilde}

\newcommand{\supp}[1]{\> \operatorname{supp}\> #1}
\newcommand{\Supp}[2]{\supp{#1}\subset #2}
\newcommand{\shugo}[1]{\{ #1\}}
\newcommand{\Shugo}[2]{\big\{ \, #1 \, \big| \, #2 \, \big\}}
\newcommand{\LR}[1]{{\langle #1 \rangle }}
\newcommand{\chf}[1]{\textbf{1}_{#1}}

\newcommand{\eq}[2]{\begin{equation} \label{#1} \begin{split} #2 \end{split} \end{equation}}
\newcommand{\eqq}[1]{\begin{equation*} \begin{split} #1 \end{split} \end{equation*}}

\newcommand{\norm}[2]{\big\| #1 \big\| _{#2}}
\newcommand{\tnorm}[2]{\| #1 \| _{#2}}
\newcommand{\hx}{\hspace{10pt}}

\newcommand{\eqs}[1]{\begin{gather*} #1 \end{gather*}}

\title[Zakharov system on torus]{Local well-posedness for the Zakharov system on multidimensional torus}
\author[Nobu Kishimoto]{Nobu Kishimoto}
\address{Department of Mathematics, Kyoto University}
\email{n-kishi@math.kyoto-u.ac.jp}
\date{}

\begin{document}

\begin{abstract}
The initial value problem of the Zakharov system on two dimensional torus with general period is shown to be locally well-posed in the Sobolev spaces of optimal regularity, including the energy space.
Unlike the one dimensional case studied by Takaoka (1999), the optimal regularity does not depend on the period of torus.
Proof relies on a standard iteration argument using the Bourgain norms.
The same strategy is also applicable to three and higher dimensional cases.
\end{abstract}

\maketitle


\section{Introduction}\label{sec_intro}
In the present paper, we investigate the initial value problem of the Zakharov system with periodic boundary condition:
\begin{equation}\label{Z}
\left\{
\begin{array}{@{\,}r@{\;}l}
i\p _tu+\Delta _\al u&=\la nu,\qquad u:[-T,T]\times \Tg ^d \to \Bo{C},\\
\frac{1}{c_0^2}\p _t^2n-\Delta _\al n&=\Delta _\be (|u|^2),\qquad n:[-T,T] \times \Tg ^d\to \R ,\\
(u,n,\p _tn)\big| _{t=0}&=(u_0,n_0,n_1)\in H^{s,l},
\end{array}
\right.
\end{equation}
where
\eq{constants}{\la \in \Bo{C}\setminus \shugo{0},\quad c_0\in \R _+:=(0,\I ),\quad \al ,\ga \in \R _+^d,\quad \be \in \R ^d\setminus \shugo{(0,\dots ,0)}}
are constants,
\[ \Delta _\al:=\al _1\frac{\p ^2}{\p x_1^2}+\dots +\al _d\frac{\p ^2}{\p x_d^2}\]
denotes the Laplacian with general coefficients $\al$, and
\[ \Tg ^d:=(\R /2\pi \ga _1\Bo{Z})\times \dots \times (\R /2\pi \ga _d\Bo{Z})\]
denotes the $d$ dimensional torus of general period $2\pi \ga$.
For a $2\pi \ga$-periodic function $\fy$, we define the Fourier coefficients $\F _x\fy (k)\equiv\hhat{\fy}(k)$ by
\[ \hhat{\fy}(k):=\int _{\Tg ^d}e^{-ik\cdot x}\fy (x)\,dx,\qquad k\in \Zg ^d:=(\ga _1^{-1}\Bo{Z})\times \dots \times (\ga _d^{-1}\Bo{Z}).\]
We also define the spacetime Fourier transform of a function $u(t,x)$ on $\R \times \Tg ^d$ in the usual fashion, denoted by $\F _{t,x}u(\tau ,k)\equiv \ti{u}(\tau ,k)$.
Then, the spaces of initial data
\eqq{H^{s,l}:=H^{s}(\Tg ^d;\Bo{C})\times H^{l}(\Tg ^d;\R )\times H^{l-1}(\Tg ^d;\R )}
for $s,l\in \R$ are the Sobolev spaces on $\Tg ^d$ equipped with the norm
\eqs{\norm{(u_0,n_0,n_1)}{H^{s,l}}^2:=\norm{u_0}{H^{s}}^2+\norm{n_0}{H^{l}}^2+\norm{n_1}{H^{l-1}}^2,\\
\norm{\fy}{H^s(\Tg ^d)}^2:=\frac{1}{\ga _1\dots \ga _d}\sum _{k\in \Zg ^d}\LR{k}^{2s}|\hhat{\fy}(k)|^2,\qquad \LR{k}:=(1+|k|^2)^{1/2}.}

This equation, introduced by Zakharov~\cite{Z}, is a mathematical model for the \mbox{Langmuir} turbulence in unmagnetized ionized plasma; $u$ represents the slowly varying envelope of rapidly oscillating electric field, and $n$ is the deviation of ion density from its mean.
It is natural from the physical point of view to consider spatially anisotropic Laplacians or tori.
Nevertheless, we can normalize constants as $c_0=1$, $\al =(1,\dots ,1)$ by a spacetime scaling.
In this article, all of these constants are supposed to be fixed and we will not consider parameter limits such as $c_0\to \I$ and $|\ga _j|\to \I$.

We study the local well-posedness of the initial value problem \eqref{Z} in $H^{s,l}$.
Here, the local well-posedness in $H^{s,l}$ means the existence of local-in-time strong solutions belonging to the class
\eqs{u\in C([-T,T];H^{s}(\Tg ^d;\Bo{C})),\\
n\in C([-T,T];H^{l}(\Tg ^d;\R ))\cap C^1([-T,T];H^{l-1}(\Tg ^d;\R ))}
(we write $(u,n)\in \mathcal{C}([-T,T];H^{s,l})$ to denote this for short), uniqueness of solutions in a suitable function space, and continuous dependence of solutions upon initial data.
The aim of this article is to establish the above properties in as low regularity as possible.
To prove these properties, we shall use the basic Bourgain method (\cite{B0}), namely, an iteration argument for the integral equations corresponding to the initial value problem \eqref{Z} using the Bourgain norms to be defined later.
One of the motivations for pushing down the regularity is to construct strong solutions in the regularity of conservation laws, such as the energy class.
For the Zakharov system, the local well-posedness in the energy class, which is roughly $H^{1,0}$, is known for the cases of $x\in \R$, $\R^2$, $\R^3$, and $\T$ (\cite{BC,GTV,Ta}), and we shall prove this for $\T^2$.
Another interest in the low regularity is construction of invariant measures, which was only achieved in the case of $\T$ (\cite{B94}).
We will not address this issue, however.

Well-posedness of the initial value problem for the Zakharov system has been extensively studied for the nonperiodic case $x\in \R ^d$.
We recall some of them here, focusing on the particular case $l=s-\frac{1}{2}$.
It is expected that the optimal (lowest) corner of the regularity range for well-posedness appears on the line $l=s-\frac{1}{2}$, because in this case two equations in the Zakharov system equally share the loss of derivative.
In addition, the ``critical regularity'' with respect to scaling, which is $(s^c,l^c)=(\frac{d-3}{2},\frac{d-4}{2})$, is also on this line.
(The Zakharov system does not have the scaling invariance, but the concept of critical regularity was introduced in \cite{GTV} by considering some simplified system which is scaling-invariant.)
Then, the Zakharov system on $\R$ and $\R^2$ was shown to be locally well-posed in $H^{0,-\frac{1}{2}}$ by Ginibre, Tsutsumi, Velo~\cite{GTV} and by Bejenaru, Herr, Holmer, Tataru~\cite{BHHT}, respectively.
In these cases $H^{0,-\frac{1}{2}}$ is known to be the lowest regularity that can be achieved by the direct iteration method, although it is away from the scaling-critical regularity.
In higher dimensional cases, the local well-posedness was established in the whole subcritical range, namely, in $H^{s^c+\e ,l^c+\e}$ with any $\e >0$, in \cite{GTV} for $d\ge 4$ and by Bejenaru, Herr~\cite{BH} for $d=3$.
The well-posedness in the energy class $H^{1,0}$, which is not on the line $l=s-\frac{1}{2}$, was obtained by Bourgain, Colliander~\cite{BC} for $d=2,3$ and in~\cite{GTV} for $d=1$.

Compared to this, there are few results on the periodic boundary value problem.
As far as the author knows, no well-posedness result for $d\ge 2$ has been found so far in the literature.
However, in the case of one spatial dimension, the sharp local well-posedness was given by Takaoka \cite{Ta}, which is again in the lowest regularity achieved by the direct iterative approach.
In \cite{Ta} it was shown that \eqref{Z} is locally well-posed in $H^{0,-1/2}$ when $\ga \not\in \Bo{N}$ and in $H^{1/2,0}$ when $\ga \in \Bo{N}$, both of which are sharp.
Note that the best regularity depends on the spatial period.
We also remark that an invariant Gibbs measure was constructed by Bourgain~\cite{B94} in the one dimensional, $\ga =1$ case.

Our main results address the case $d\ge 2$.
It seems interesting that the period $\ga$ has nothing to do with the regularity threshold, in contrast to the 1d case.
\begin{thm}\label{thm_main}
Let $d\ge 2$.
Then, for \emph{any} $\ga \in \R _+^d$, \eqref{Z} is locally well-posed in $H^{s,l}$ with $(s,l)$ in the range
\begin{alignat}{3}
&0\le s-l\le 1,&\quad &2s\ge l+\tfrac{d}{2}>d-1,&\qquad &(\text{for}\hx d\ge 3)\label{hani_3d}\\
&0\le s-l\le 1,&\quad &2s\ge l+1\ge 1.&\qquad &(\text{for}\hx d=2)\label{hani_2d}
\end{alignat}
\end{thm}
The precise statement for well-posedness results will be given in Theorem~\ref{thm_main2} after introducing function spaces.
Next, we give negative results.
\begin{thm}\label{thm_illp}
For any $d\ge 2$ and $\ga \in \R _+^d$, the data-to-solution map of \eqref{Z} on smooth data cannot extend to a $C^2$ map from any neighborhood of the origin in $H^{s,l}$ into $\mathcal{C}([-T,T];H^{s,l})$ for any $T>0$, provided $l>\min \shugo{2s-1,\,s+1}$ or $l<\max \shugo{0,\,s-2}$.
Moreover, if $d=2$ and $s<\frac{3}{2}$, $l\ge 0$, $l>2s-1$, then \eqref{Z} is ill-posed in $H^{s,l}$.
\end{thm}

See Figure~\ref{fig_regularity} for the ranges of regularity in these theorems.
\begin{figure}
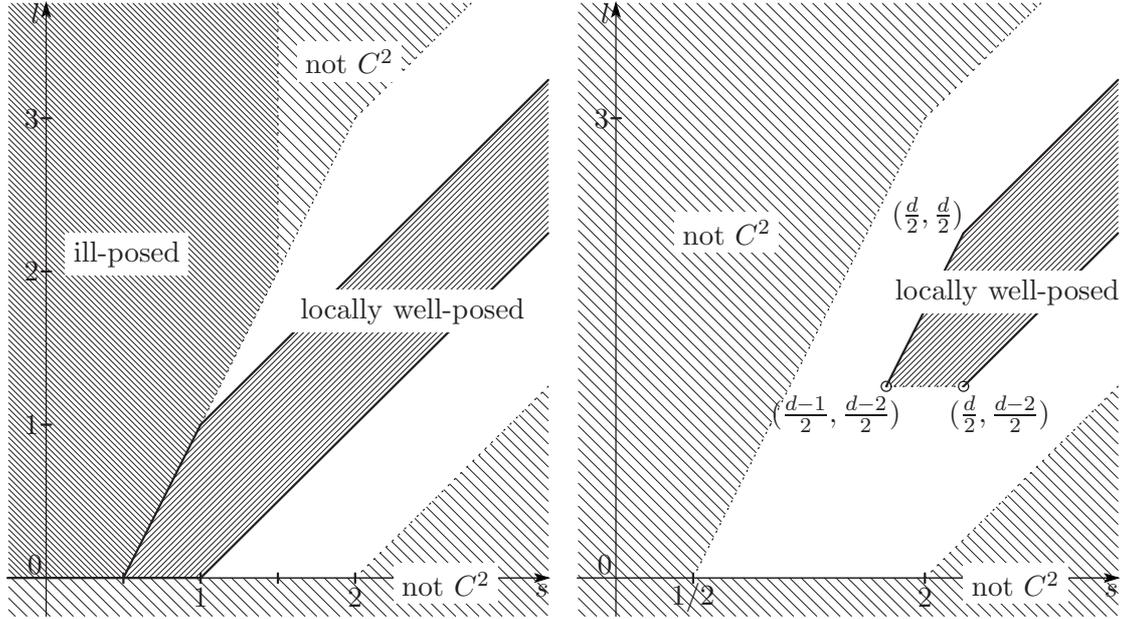
\label{fig_regularity}
\begin{center}
\input{Regularity2d.tex}\hspace{-0pt}\input{Regularity3d.tex}
\end{center}
\caption{Regularity assumptions in Theorems~\ref{thm_main} and \ref{thm_illp}, $d=2$ (left) and $d\ge 3$ (right).
The phrase ``not $C^2$'' means that there does not exist the data-to-solution map which is an extension of the map for smooth data and in $C^2$ with respect to the $H^{s,l}$ norm.}
\end{figure}
We remark that in 2d there is no gap between the regularity ranges indicated in Theorems~\ref{thm_main} and \ref{thm_illp}, at least on the line $l=s-\frac{1}{2}$.
In this sense we can say that our local well-posedness result for 2d is optimal.
For the results of $d\ge 3$, there still remains some gaps between the ranges given in two theorems above.
Theorem~\ref{thm_illp} will be given as a part of Theorem~\ref{thm_illp2} below.

Under some condition on parameters, \eqref{Z} is also described as a Hamiltonian PDE assuming that the initial velocity has zero mean, namely $\hhat{n_1}(0)=0$.
In the case where $c_0=\la =1$ and $\al =\be =(1,\dots ,1)$, the Hamiltonian is given by
\eqq{H(u,n)(t):=\norm{\nabla u(t)}{L^2}^2+\frac{1}{2}(\norm{n(t)}{L^2}^2+\norm{|\nabla |^{-1}\p _tn(t)}{L^2}^2)+\int _{\T ^d_\ga}n(t,x)|u(t,x)|^2\,dx.}
The energy space is then $H^1\times L^2\times |\nabla|L^2$, which is a closed subspace of $H^{1,0}$ invariant under the flow.
Since Theorem~\ref{thm_main} implies the local well-posedness in $H^{1,0}$ in the 2d case, using the conservation law of the Hamiltonian we obtain the global well-posedness in the energy space in 2d under some smallness assumption.
\begin{cor}
Let $d=2$, $\al ,\be \in \R _+^d$, $c_0,\la >0$.
Assume that $\al$ and $\be$ are linearly dependent.
Then, for initial data $(u_0,n_0,n_1)\in H^{1,0}$ satisfying $\hhat{n_1}(0)=0$ and $\tnorm{u_0}{L^2}\ll 1$, the solution to \eqref{Z} exists globally in time.
\end{cor}
The global existence of solution and the blow-up problem for the Zakharov system on $\T ^2$ will be discussed further in our forthcoming paper~\cite{KiM}.
In particular, it will turn out that we do not have to assume $\hhat{n_1}(0)=0$, and that the assumption $\tnorm{u_0}{L^2}\ll 1$ can be replaced with $\tnorm{u_0}{L^2(\T ^2)}\le \tnorm{Q}{L^2(\R ^2)}$, which is the optimal threshold in the sense that there exists a finite-time blow-up solution starting from an initial datum with $\tnorm{u_0}{L^2}$ greater than but arbitrarily close to $\tnorm{Q}{L^2(\R ^2)}$.
Here, $Q$ denotes the ground state solution to the focusing cubic nonlinear Schr\"odinger equation on $\R^2$.
These results are the periodic counterpart of the results on $\R ^2$ given by Glangetas and Merle~\cite{GM1,GM2}.

The plan of this article is as follows.
In Section~\ref{sec_settings}, we will define the Bourgain spaces and prepare some fundamental estimates.
Using them, we will prove a variety of trilinear estimates in Section~\ref{sec_trilinear}, which will be combined to establish Theorem~\ref{thm_main} in Section~\ref{sec_pf-main}.
Finally, we will give a proof of Theorem~\ref{thm_illp} in Section~\ref{sec_pf-illp}.

In the rest of this section, we take a brief look at our problem and strategies.
Throughout this article, we write $A\lesssim B$ to denote the estimate $A\le CB$ with a constant $C>0$, which may depend on some parameters in a harmless way, and denote $A\lesssim B\lesssim A$ by $A\sim B$.
$A\gg B$ means that $A\ge CB$ with some sufficiently large constant $C>0$.
Also, we use the notation $a+$ or $a-$ for $a\in \R$ to denote $a+\e$ or $a-\e$, respectively, with $\e>0$ arbitrarily small.

Let us first see the difference of our case from the 1d problem treated in \cite{Ta}.
For the sake of simplicity, we assume $c_0=\la =1$ and $\al =\be =(1,\dots ,1)$.
The initial value problem is replaced by the following system of integral equations:
\eqq{
u(t)&=e^{it\Delta}u_0-i\int _0^te^{i(t-t')\Delta}\big[ n(t')u(t')\big] \,dt',\\
n(t)&=\cos (t|\nabla |)n_0+\frac{\sin (t|\nabla |)}{|\nabla |}n_1+\int _0^t\sin ((t-t')|\nabla |)|\nabla |\big[ u(t')\bbar{u(t')}\big] \,dt'.
}

The standard iteration method requires a suitable control of all the iteration term by the initial data.
For instance, the quadratic iteration term for the wave equation:
\eqq{\int _0^t \sin ((t-t')|\nabla |)|\nabla |\big[ e^{it'\Delta}u_0\cdot \bbar{e^{it'\Delta}u_0}\big] \,dt'}
would be controlled in $H^l$ by the $H^s$ norm of initial data $u_0$ if we could obtain a local solution by the iteration method in $H^{s,l}$.
The Fourier coefficient of it at $k$ is calculated as
\begin{align}
&c\int _0^t \sin (|k|(t-t'))|k|\frac{1}{|\ga |}\sum _{k'\in \Zg ^d}e^{-i|k-k'|^2t'}\hhat{u_0}(k-k')e^{i|k'|^2t'}\bbar{\hhat{u_0}(-k')}\,dt'\notag \\[-5pt]
&=c|k|\sum _{\sgm =\pm 1}\sgm e^{i\sgm |k|t}\frac{1}{|\ga |}\sum _{k'\in \Zg ^d}\hhat{u_0}(k-k')\bbar{\hhat{u_0}(-k')}\int _0^te^{i\{ -\sgm |k|-|k-k'|^2+|k'|^2\} t'}dt',\label{second}
\end{align}
where $|\ga |:=\ga _1\ga _2\dots \ga _d$.
The term $|k|$ in \eqref{second} indicates one derivative loss.
However, note that the integral in \eqref{second} is bounded by $2/M$, where
\eqq{M:=\big| -\sgm |k|-|k-k'|^2+|k'|^2\big| .}
If $M$ is sufficiently large, then we can cancel (some of) the derivative loss $|k|$ with this integral (\emph{non-resonant} case).
On the other hand, if $M$ is small (especially if $M\lec 1$), we can gain no derivative to cancel $|k|$ with (\emph{resonant} case).
As we will see in the proof of Theorem~\ref{thm_illp2}, the resonance phenomenon in the quadratic iteration term plays an essential role in determining the optimal regularity for the iteration method.

In 1d, $M$ can be rewritten as
\eqq{M&=|k|\Big| -\sgm +|2k'-k| \mathrm{sgn}\big( k(2k'-k)\big) \Big| \\
&=\ga ^{-1}|k|\Big| -\sgm\ga +\ga |2k'-k| \mathrm{sgn}\big( k(2k'-k)\big) \Big| .}
Note that the second absolute value of the right hand side can be equal to zero for some $k,k'\in \ga ^{-1}\Bo{Z}$ ($|k|$ large) if and only if $\ga \in \Bo{N}$.
In the case of $\ga \in \Bo{N}$, we consider, for example, $k=2N-1$ and $k'=N$ for an arbitrary $N\in \Bo{N}$, then the resonance will happen when $\sgm =+1$.
On the other hand, when $\ga \not\in \Bo{N}$, we easily verify that
\eqq{M\ge c(\ga )|k|,\qquad c(\ga ):=\ga ^{-1}\mathrm{dist}(\ga ,\Bo{Z})>0}
for any $k,k'$, which means that the integral in \eqref{second} provides enough gain of derivative so that one derivative loss can be totally cancelled out.
In this case, the optimal regularity $H^{0,-1/2}$ given in \cite{Ta} is actually determined by the resonance in the cubic iteration term.

Situation is totally different in our higher dimensional cases, where
\eqq{M=|k|\Big| -\sgm +|2k'-k|\cos \angle (k,2k'-k)\Big| .}
Now, we can make $M\lec 1$ with some $|k|\gg 1$ by exploiting flexibility of the angle between $k$ and $2k'-k$, even if $\ga _j\not\in \Bo{N}$.
In fact, for an arbitrary (large) $N\in \Bo{N}$, set
\eqq{k=(\frac{2N-1}{\ga _1},-\frac{1}{\ga _2},0,\dots ,0),\quad  k'=(\frac{N}{\ga _1},\frac{n-1}{\ga _2},0,\dots ,0)}
with $n\in \Bo{Z}$ to be chosen momentarily.
A direct calculation shows that
\eqq{M=\bigg| -\sgm \sqrt{\frac{(2N-1)^2}{\ga _1^2}+\frac{1}{\ga _2^2}}+\frac{2N-1}{\ga _1^2}+\frac{1}{\ga _2^2}-\frac{2n}{\ga _2^2}\bigg| .}
Then, for any $\ga$ and fixed $N$ we can choose $n\in \Bo{Z}$ so that $M\le \ga _2^{-2}$.
Therefore, resonance can happen for any period $\ga$ in higher dimensional cases, as suggested in our theorems.

We finally explain our strategy to prove the crucial nonlinear estimates.
Following an approach taken in~\cite{BHHT} for the case of $\R ^2$, we will use the Bourgain spaces of $\ell ^1$-Besov type with respect to modulations as the spaces for iteration (see Section~\ref{sec_settings} for its definition).
This structure, originally introduced in~\cite{Tat}, easily reduces the nonlinear estimates to the corresponding estimates for functions restricted dyadically with respect to the size of frequency and modulation.
These ``block estimates,'' as systematically treated in \cite{T01}, will be shown via the Cauchy-Schwarz inequality and the estimate on the amount of contributing frequency pairs in the nonlinear interaction between considered frequency blocks.
In some cases, as done in \cite{BHHT}, we have to employ some finer decompositions of functions with respect to the angle of frequency.
Then, all the estimates will be combined to yield estimates on functions with no restriction, via the Cauchy-Schwarz inequality together with the $\ell ^1$-Besov nature of the spaces and some orthogonality properties.

In the nonperiodic case, the total volume of contributing frequencies in each ``block estimate'' can be often measured simply by the Jacobian determinant of an appropriate change of variables; see \cite{CDKS}, for instance.
The reason behind it is that the volume is a continuous quantity.
In fact, similar results are expected for the problem on the mixed space $\R \times \T$; see \mbox{e.g.} \cite{TaTz} for a result of this direction.

In the purely periodic case, however, the total number of contributing frequencies is discrete and the change of variables argument can be applicable only in the restricted situations.
This is closely related to the less dispersive nature of the periodic problem; recall that the local smoothing estimate for the linear Schr\"odinger evolution, which gains half a derivative in space, totally fails in the periodic setting.
Hence, the ``block estimates'' for the $\R^2$ case \cite{BHHT} cannot be extended to the $\T^2$ case in any obvious manner, indeed our result on $\T^2$ is half a regularity worse than the result on $\R^2$ but optimal in the sense of Theorem~\ref{thm_illp}.

We will obtain sharp upper bounds on the total number of contributing frequencies through careful geometric observations and orthogonality arguments.
As mentioned above, in the case of stronger resonance (\mbox{i.e.} smaller value of $M$) we can only expect weaker nonlinear smoothing effect.
Consider, for example, the nonlinear interaction in which the wave frequency $k$ is produced by two Schr\"odinger frequencies $k-k'$, $k'$ satisfying the relation
\[ M=\big| -\sgm |k|-|k-k'|^2+|k'|^2\big| \sim L\]
for some $L\ll |k|$ (thus the nonlinear smoothing effect is not as strong as one derivative).
We easily see that $|k|\lec |k-k'|\sim |k'|$.
In addition, the above implies the following two relations:
\begin{gather}
\big| |k-k'|-|k'|\big| \lec \frac{L+|k|}{|k-k'|+|k'|}\lec 1,\label{resonance1}\\
\frac{k}{|k|}\cdot k'=\frac{|k|^2+\sgm |k|}{2|k|}+O(\frac{L}{|k|}).\label{resonance2}
\end{gather}
From \eqref{resonance1} we see that the difference between sizes of two Schr\"odinger frequencies is always bounded by $1$, which (by an orthogonality argument) allows us to restrict both frequencies to an annulus of thickness $\lec 1$ centered at the origin.
On the other hand, \eqref{resonance2} shows that $k'$ is confined to some plate-like region of thickness $\sim \frac{L}{|k|}\lec 1$ for fixed $k$.
As a consequence, for fixed $k$, the Schr\"odinger frequency $k'$ which contributes through the nonlinear interaction is confined to a small region described as the intersection of a thin annulus and a thin plate.
We will carefully count the number of such $k'$, which will turn out to be small enough to countervail the lack of smoothing effect, obtaining the desired nonlinear estimates.

We also note that the difference of $\frac{1}{2}$ regularity between results on $\R^2$ and $\T^2$ is due to the estimates for (nearly) resonant frequencies; Propositions~\ref{prop_te_highhigh-m}, \ref{prop_te_highhigh-l}, and \ref{prop_te_highhigh-l_2d} below.
In fact, these estimates require $\frac{1}{2}$ more derivative compared to the $\R^2$ case; Propositions 4.4, 4.6, and 4.7 in \cite{BHHT}.
If we consider instead the ``high-low interaction'' case $|k-k'|\ll |k'|$ or $|k-k'|\gg |k'|$, where it holds $M\sim |k|^2$ (thus non-resonant), then we can obtain an estimate for $\T^2$ similar to (even better than) that for $\R^2$; compare Proposition~\ref{prop_te_highlow2} below with Proposition~4.8 in \cite{BHHT}.

We conclude this section by making one more remark.
In the multidimensional periodic setting, some number theoretic arguments are often employed in counting the number of lattice points belonging to a particular frequency region.
One of such tricks are the following estimate:
\eq{est_nt}{\# \Shugo{k\in \Bo{Z}^d}{N\le |k|^2<N+1}\lec N^{\frac{d-2}{2}}\exp \frac{c\log N}{\log \log N},\qquad d\ge 2,\quad N\gg 1,}
which has been repeatedly used since the work of Bourgain~\cite{B0}.
It seems not so easy, however, to harmonize the estimate of this type with our geometric considerations.
In the proof of the nonlinear estimates we will never quote such number theoretic tricks.
Instead, when counting the number of frequencies, we will simply use the fact that there are at most $O(1)$ (depending on $\ga$, $d$) lattice points of $\Zg ^d$ in any ball $\subset \R ^d$ of unit size.
Hence, our argument can be applied to the case of any spatial period $\ga$.
Note that the above estimate \eqref{est_nt} is not known if we replace $\Bo{Z}^d$ with ``irrational tori'' $\Zg ^d$, as mentioned in \cite{B07}.
Some Strichartz-type inequalities related to the Schr\"odinger equation on irrational tori were also obtained in \cite{B07}.


\section{Settings, Preliminaries}\label{sec_settings}

We turn to the details of the well-posedness theory.
For the reader's convenience, some of the notations introduced below are the same as those used in \cite{BHHT}.

It is convenient to reduce the original Zakharov system to a first-order system by putting
\[ w:=n+i\LR{\nabla}^{-1}\p _t n,\qquad w_0:=n_0+i\LR{\nabla}^{-1}n_1.\]
Here, $\LR{\nabla}$ denotes the spatial Fourier multiplier corresponding to $\LR{k}$.
The new system is then given by
\begin{equation}\label{Z'}
\left\{
\begin{array}{@{\,}r@{\;}l}
i\p _tu+ \Delta u&=\frac{\la }{2}(w+\bar{w})u,\qquad (t,x)\in [-T,T] \times \Tg ^d,\\[5pt]
i\p _tw-\LR{\nabla}w&=-\LR{\nabla}^{-1}\Delta _\be (|u|^2)-\LR{\nabla}^{-1}\frac{w+\bar{w}}{2},\\[5pt]
(u,w)\big| _{t=0}&=(u_0,w_0)\in H^{s}(\Tg ^d;\Bo{C})\times H^{l}(\Tg ^d;\Bo{C}).
\end{array}
\right.
\end{equation}
Note that we have normalized constants as $\al =(1,\dots ,1)$, $c_0=1$.
Since $n$ is real-valued, we can recover the solution to \eqref{Z} from a given solution $(u,w)$ to \eqref{Z'} by letting $n$ be the real part of $w$.
The following is the corresponding system of integral equations.
\eq{IE}{u(t)&=e^{it\Delta}u_0-\frac{i\la }{2}\int _0^te^{i(t-t')\Delta}\big[ (w+\bar{w})u\big] (t')\,dt',\\
w(t)&=e^{-it\LR{\nabla}}w_0+i\int _0^t e^{-i(t-t')\LR{\nabla}}\big[ \frac{\Delta _\be}{\LR{\nabla}}(u\bar{u})+\frac{1}{2\LR{\nabla}}(w+\bar{w})\big] (t')\,dt'.}

As we have seen in Section~\ref{sec_intro}, serious resonances can occur without regard to the period $\ga$ in the higher dimensional cases.
However, most of the frequency pairs $(k,k')$ are non-resonant.
In establishing nonlinear estimates, it will be important to make best use of the nonlinear smoothing effect which comes from the oscillations of solutions and nonlinear interactions between such non-resonant frequency pairs.
It is well-known that the Bourgain norms are very well fit for this purpose.

\begin{defn}[Littlewood-Paley decomposition]\label{def_LPdec}
Let $\eta \in C_0^\I (\R )$ be an even function with the properties
\eqq{\eta \equiv 1\hx \text{on}\hx [-1,1],\quad \Supp{\eta}{(-2,2)},\quad 0\le \eta \le 1.}
Define a partition of unity on $\R$, $\eta _N$ for dyadic $N\ge 1$, by
\eqq{\eta _1:=\eta ,\qquad \eta _N(r):=\eta (\tfrac{r}{N})-\eta (\tfrac{2r}{N}),\quad N\ge 2.}
Define the frequency localization operator $P_N$ on functions $\fy :\Tg ^d\to \Bo{C}$ by
\eqq{\F _x(P_N\fy )(k):=\eta _N(|k|)\hhat{\fy}(k).}
We also use the notation $P_N$ to denote the operator on functions in $(t,x)$,
\eqq{\F _{x}(P_Nu)(t,k):=\eta _N(|k|)\hhat{u}(t,k).}
Also, define the operators $Q_L^S$, $Q_L^{W_\pm}$ on spacetime functions by
\eqs{\F _{t,x}(Q_L^Su)(\tau ,k):=\eta _L(\tau +|k|^2)\ti{u}(\tau ,k),\quad \F _{t,x}(Q_L^{W_\pm}w)(\tau ,k):=\eta _L(\tau \pm |k|)\ti{w}(\tau ,k)}
for dyadic numbers $L\ge 1$.
We will write $P^S_{N,L}=P_NQ^S_L$, $P^{W_{\pm}}_{N,L}=P_NQ^{W_{\pm}}_L$ for brevity.
Finally, we define several dyadic frequency regions:
\eqs{\FR{P}_1:=\Shugo{(\tau ,k)}{|k|\le 2},\quad \FR{P}_N:=\Shugo{(\tau ,k)}{\tfrac{N}{2}\le |k|\le 2N},\quad N\ge 2,\\
\FR{S}_1:=\Shugo{(\tau ,k)}{|\tau +|k|^2|\le 2},\quad \FR{S}_L:=\Shugo{(\tau ,k)}{\tfrac{L}{2}\le |\tau +|k|^2|\le 2L},\quad L\ge 2,\\
\FR{W}^\pm _1:=\Shugo{(\tau ,k)}{|\tau \pm |k||\le 2},\quad \FR{W}^\pm _L:=\Shugo{(\tau ,k)}{\tfrac{L}{2}\le |\tau \pm |k||\le 2L},\quad L\ge 2.}
In what follows, capital letters $N$ and $L$ are always used to denote dyadic numbers $\ge 1$.
We will often use these capital letters with various subscripts, and also the notation
\eqq{\overline{N}_{ij\dots}:=\max \shugo{N_i,N_j,\dots},\qquad \underline{N}_{ij\dots}:=\min \shugo{N_i,N_j,\dots}.}
The following will be used for the specific indices;
\eqq{N_{\max}:=\overline{N}_{012},\quad N_{\min}:=\underline{N}_{012},\qquad L_{\max}:=\overline{L}_{012},\quad L_{\min}:=\underline{L}_{012}.}
\end{defn}

\begin{defn}[Bourgain spaces]
For $s,b\in \R$ and $1\le p<\I$, define the Bourgain space for the Schr\"odinger equation $X^{s,b,p}_S$ and that for reduced wave equations $X^{s,b,p}_{W_{\pm}}$ by the completion of functions $C^\I$ in space and Schwartz in time with respect to
\eqq{\norm{u}{X^{s,b,p}_S}&:=\norm{\norm{\shugo{N^sL^b\norm{P_{N,L}^Su}{L^2_{t,x}(\R \times \Tg ^d)}}_{N,L}}{\ell ^p_L}}{\ell ^2_N},\\
\norm{u}{X^{s,b,p}_{W_{\pm}}}&:=\norm{\norm{\shugo{N^sL^b\norm{P_{N,L}^{W_{\pm}}u}{L^2_{t,x}(\R \times \Tg ^d)}}_{N,L}}{\ell ^p_L}}{\ell ^2_N}.}
We also define the Bourgain space for the wave equation $X^{s,b,p}_W$ by setting
\[ Q^W_L:=\F ^{-1}_{\tau ,k}\eta _L(|\tau |-|k|)\F _{t,x}\]
and replacing $W_\pm$ with $W$ in the above definition of $X^{s,b,p}_{W_{\pm}}$.
For $T>0$, define the restricted space $X^{s,b,p}_*(T)$ ($*=S$ or $W_{\pm}$ or $W$) by the restrictions of distributions in $X^{s,b,p}_*$ to $(-T,T)\times \Tg ^d$, with the norm
\eqq{\norm{u}{X^{s,b,p}_*(T)}:=\inf \Shugo{\norm{U}{X^{s,b,p}_*}}{U\in X^{s,b,p}_*~\text{is an extension of $u$ to $\R \times \Tg ^d$}}.}
\end{defn}

Theorem~\ref{thm_main} is then precisely stated as follows.
\begin{thm}\label{thm_main2}
Let $d\ge 2$, $\la ,c_0,\al,\be,\ga$ be any constants as \eqref{constants}, and let $(s,l)\in \R ^2$ satisfy \eqref{hani_3d} or \eqref{hani_2d}.
Then, for any $r>0$, there exists a time $T\gec \min \shugo{1,r^{-2-}}$ such that for any initial data $(u_0,n_0,n_1)$ in $H^{s,l}$ with norm less than $r$, there exists a unique solution $(u,n)$ to \eqref{Z} in the class
\[ (u,n)\in X^{s,\frac{1}{2},1}_S(T)\times X^{l,\frac{1}{2},1}_W(T),\qquad \p _tn\in X^{l-1,\frac{1}{2},1}_W(T),\]
which is continuously embedded into $\mathcal{C}([-T,T];H^{s,l})$.
Moreover, the map $(u_0,n_0,n_1)\mapsto (u,n)$ is Lipschitz continuous as a map from the ball in $H^{s,l}$ into the class defined above.
\end{thm}
To prove Theorem~\ref{thm_main2}, it suffices to show similar statements on the reduced system \eqref{Z'}; see \cite{BHHT} for details.
\begin{thm}\label{thm_main3}
Let $d\ge 2$, $\la \in \mathbb{C}\setminus \shugo{0}$, $\be \in \R ^d\setminus \shugo{(0,\dots ,0)}$, and let $\ga \in \R _+^d$ be any period.
Assume that $(s,l)\in \R ^2$ satisfies \eqref{hani_3d} or \eqref{hani_2d}.
Then, for any $r>0$, there exists a time $T\gec \min \shugo{1,r^{-2-}}$ such that for any initial data $(u_0,w_0)\in H^s\times H^l$ with its norm less than $r$, there exists a unique solution $(u,w)\in X^{s,\frac{1}{2},1}_S(T)\times X^{l,\frac{1}{2},1}_{W_+}(T)$ to \eqref{Z'}.
Moreover, the map $(u_0,w_0)\mapsto (u,w)$ is Lipschitz continuous as a map from the ball in $H^s\times H^l$ into $X^{s,\frac{1}{2},1}_S(T)\times X^{l,\frac{1}{2},1}_{W_+}(T)$.
\end{thm}
Clearly, we call $(u,n)$ a solution to \eqref{Z} if $(u,n+i\LR{\nabla}^{-1}\p _tn)$ is a solution to the integral equation \eqref{IE} after the normalization of constants.
The proof of Theorem~\ref{thm_main3} will be given in Section~\ref{sec_pf-main} with some linear and bilinear estimates.

In the rest of this section, we prepare some preliminary lemmas.
The following is a periodic analog of a bilinear refinement of $L^4$-Strichartz estimate in the $\R ^d$ case as well as similar estimates for the Schr\"odinger-wave interactions (\cite{BHHT,B98}).
Here and in the sequel we write $\ze =(\tau ,k)$ and $\int _\ze\cdots =\int _{\tau \in \R}\frac{1}{|\ga |}\sum _{k\in \Zg ^d}\cdots$.

\begin{lem}[Bilinear Strichartz estimates]\label{lem_bs}
Let $d\ge 2$ and $N_j, L_j\ge 1$ ($j=0,1,2$) be dyadic numbers.

(i) Suppose that $u_1,u_2\in L^2(\R \times \Tg ^d)$ satisfy
\eqq{\Supp{\ti{u_1}}{\FR{P}_{N_1}\cap \FR{S}_{L_1}},\qquad \Supp{\ti{u_2}}{\FR{P}_{N_2}\cap \FR{S}_{L_2}}.}
We also assume $N_0\ge 2$.
Then we have
\eqq{\norm{\ti{u_1\bbar{u_2}}}{L^2_{\ze}(\FR{P}_{N_0})}\lec \underline{L}_{12}^{\frac{1}{2}}\Big( \frac{\overline{L}_{12}}{N_0}+1\Big) ^{\frac{1}{2}}N_{\min}^{\frac{d-1}{2}}\norm{u_1}{L^2_{t,x}}\norm{u_2}{L^2_{t,x}}.}

(ii) Suppose that $u,w\in L^2(\R \times \Tg ^d)$ satisfy
\eqq{\Supp{\ti{w}}{\FR{P}_{N_0}\cap \FR{W}^{\pm}_{L_0}},\qquad \Supp{\ti{u}}{\FR{P}_{N_1}\cap \FR{S}_{L_1}}.}
Then we have
\eqq{\norm{\ti{wu}}{L^2_{\ze}(\FR{P}_{N_2})}+\norm{\ti{\bar{w}u}}{L^2_{\ze}(\FR{P}_{N_2})}\lec \underline{L}_{01}^{\frac{1}{2}}\Big( \frac{\overline{L}_{01}}{N_1}+1\Big) ^{\frac{1}{2}}N_{\min}^{\frac{d-1}{2}}\norm{w}{L^2_{t,x}}\norm{u}{L^2_{t,x}}.}
\end{lem}

\begin{rem}
The implicit constants in the above estimates depend only on $d$ and $\ga$.
Here and in what follows we omit to specify dependence of constants on the dimension or the spatial period.
\end{rem}
\begin{rem}
For (i), a similar estimate was obtained in the case $d=2$ by De~Silva, Pavlovi\'c, Staffilani, and Tzirakis~\cite{dSPST}.
Their result (Proposition~4.6~(a) in \cite{dSPST}) reads in our setting as follows: if moreover $N_1\gg N_2$ or $N_1\ll N_2$ and the period $\ga =(\ga _1,\ga _2)$ satisfies $\ga _1=\ga _2$, then
\eqq{\big( \norm{u_1\bbar{u_2}}{L^2_{t,x}}=\big) ~\norm{u_1u_2}{L^2_{t,x}}\lec L_1^{\frac{1}{2}}L_2^{\frac{1}{2}}\underline{N}_{12}^{0+} \norm{u_1}{L^2_{t,x}}\norm{u_2}{L^2_{t,x}}.}
Compared to this, our estimate, which has a prefactor $L_1^{\frac{1}{2}}L_2^{\frac{1}{2}}(\frac{N_{\min}}{N_0}+\frac{N_{\min}}{\overline{L}_{12}})^{\frac{1}{2}}$ at the cost of restriction of frequency onto $\FR{P}_{N_0}$ and limitation to the specific bilinear form of $u_1\bbar{u_2}$, is verified by a simpler proof and applicable to the case of ``irrational tori,'' and also implies better bound when $\overline{L}_{12}\gg N_{\min}$.
Also, we remark that bilinear Strichartz estimates in \cite{dSPST} were obtained as a corollary of corresponding bilinear estimates for solutions to the linear Schr\"odinger equations, while we directly verify our bilinear estimate without using the estimate for linear solutions.
\end{rem}
\begin{rem}\label{rem_dec}
It will be clear from the proof that in (ii) we do not actually need the restriction of $\ti{wu}$ or $\ti{\bar{w}u}$ onto the dyadic region $\FR{P}_{N_2}$.
If we assume no restriction in $N_2$, the resulting estimate will be
\eqq{\norm{wu}{L^2_{t,x}}+\norm{\bar{w}u}{L^2_{t,x}}\lec \underline{L}_{01}^{\frac{1}{2}}\Big( \frac{\overline{L}_{01}}{N_1}+1\Big) ^{\frac{1}{2}}\underline{N}_{01}^{\frac{d-1}{2}}\norm{w}{L^2_{t,x}}\norm{u}{L^2_{t,x}}.}
\end{rem}

\begin{proof}
(i) Consider the case $N_{\min}=\underline{N}_{12}$ first.
We have
\eqq{&\norm{\ti{u_1\bbar{u_2}}}{L^2_{\ze}(\FR{P}_{N_0})}\sim \norm{\int \ti{u_1}(\ze _1)\bbar{\ti{u_2}}(\ze _1-\ze _0)\,d\ze _1}{L^2_{\ze _0}(\FR{P}_{N_0})}\le \sup _{\ze _0\in \FR{P}_{N_0}}\big| E(\ze _0)\big| ^{\frac{1}{2}}\norm{\ti{u_1}}{L^2_{\ze}}\norm{\ti{u_2}}{L^2_{\ze}}}
by the Cauchy-Schwarz inequality, where
\eqq{E(\tau _0,k_0):=\Shugo{(\tau _1,k_1)\in \FR{P}_{N_1}\cap \FR{S}_{L_1}}{(\tau _1-\tau _0, k_1-k_0)\in \FR{P}_{N_2}\cap \FR{S}_{L_2}}.}
Observe that if $(\tau _1,k_1)\in E(\ze _0)$, then
\eqq{\tau _1\in \big( -|k_1|^2+[-2L_1,2L_1]\big) \cap \big( \tau _0-|k_1-k_0|^2+[-2L_2,2L_2]\big) ,}
which implies
\eqq{\big| -|k_1|^2-(\tau _0-|k_1-k_0|^2)\big| \le 2(L_1+L_2),}
namely, $\tau _0-|k_0|^2+2k_0\cdot k_1=O(\overline{L}_{12})$.
Therefore, using the assumption $N_0\ge 2$, $\big| E(\ze _0)\big|$ is bounded by
\eqq{\underline{L}_{12}\cdot \Big| \Shugo{k_1\in \Zg ^d}{|k_1|\le 2N_1,\,|k_1-k_0|\le 2N_2,\,\frac{k_0}{|k_0|}\cdot k_1=-\frac{\tau _0}{2|k_0|}+\frac{|k_0|}{2}+O(\frac{\overline{L}_{12}}{N_0})}\Big| .}
We see that $k_1$ in the above set should be in a ball of size $O(\underline{N}_{12})$ with its component parallel to $k_0$ confined to an interval of length $O(\frac{\overline{L}_{12}}{N_0})$, for fixed $\ze _0$.
Such a region (in $\R ^d$) is included in the union of at most $O((\frac{\overline{L}_{12}}{N_0}+1)\underline{N}_{12}^{d-1})$ balls with radius $1$, which yields the bound
\eqq{\big| E(\ze _0)\big| \lec \underline{L}_{12}\Big( \frac{\overline{L}_{12}}{N_0}+1\Big) \underline{N}_{12}^{d-1},}
as desired.

If $N_{\min}=N_0$, divide $\ti{u_j}$, $j=1,2$, into $O((\frac{N_j}{N_0})^d)$ functions each of which is frequency localized in a cube of side length $N_0$.
Then the orthogonality admits us to reduce the estimate to the case of each component.
We can follow the above argument to obtain the desired bound.

(ii) Since $\Supp{\ti{\bar{w}}}{\FR{P}_{N_0}\cap \FR{W}^{\mp}_{L_0}}$ if and only if $\Supp{\ti{w}}{\FR{P}_{N_0}\cap \FR{W}^{\pm}_{L_0}}$, it suffices to prove the estimate for $\bar{w}u$.

Consider the case where $N_{\min}=N_0$ or $N_1$.
Similarly to (i), the claimed estimate is reduced to
\eqq{\big| E(\ze _2)\big| \lec \underline{L}_{01}\Big( \frac{\overline{L}_{01}}{N_1}+1\Big) \underline{N}_{01}^{d-1}}
for $\ze _2\in \FR{P}_{N_2}$, where
\eqq{E(\ze _2):=\Shugo{\ze _0\in \FR{P}_{N_0}\cap \FR{W}^\pm _{L_0}}{\ze _0+\ze _2\in \FR{P}_{N_1}\cap \FR{S}_{L_1}}.}
It holds for $(\tau _0,k_0)\in E(\tau _2,k_2)$ that
\eqq{\big| |k_0|-|k_2|\big| \le |k_0+k_2|\le 2N_1,\quad \tau _2\mp |k_0|+|k_0+k_2|^2=O(\overline{L}_{01}),}
which implies
\eqq{\big| E(\ze _2)\big| \lec \underline{L}_{01}\Big| \Shugo{k_0 \in \Zg ^d}{|k_0|\sim N_0,\,N_1^2\gec |k_0+k_2|^2=-\tau _2\pm |k_2|+O(\overline{L}_{01}+N_1)}\Big| .}
When $\overline{L}_{01}\gec N_1^2$ or $N_1\lec 1$, we simply replace the latter subset of $\Zg ^d$ with the set of all lattice points in a ball of radius $O(\underline{N}_{01})$, obtaining the bound
\eqq{\big| E(\ze _2)\big| \lec \underline{L}_{01}\underline{N}_{01}^d\lec \underline{L}_{01}\Big( \frac{\overline{L}_{01}}{N_1}+1\Big) \underline{N}_{01}^{d-1}.}
We thus assume $\overline{L}_{01}\ll N_1^2$ and $1\ll N_1$.
In this case, $k_0$ is confined to the intersection of a ball centered at the origin with radius $O(N_0)$ and an annulus centered at $-k_2$ with a radius of $O(N_1)$ and a width of $O(\frac{\overline{L}_{01}+N_1}{N_1})$.
Such a region (in $\R^d$) can be covered with at most $O(\frac{\overline{L}_{01}+N_1}{N_1}\underline{N}_{01}^{d-1})$ balls of radius $1$, so we reach the desired bound in the same manner.

When $N_{\min}=N_2$, we decompose $w$ and $u$ into functions frequency-supported on a cube of size $N_2$ and reduce the estimate to that for each component, then the above proof is valid with some trivial modification.
\end{proof}

Lemma~\ref{lem_bs} will be used mainly in the case of high-modulation interactions, namely $L_{\max}\gec N_{\max}^2$, where the nonlinear interactions supply enough smoothing effect.
For the lower modulation cases, however, we will have to count more carefully the number of lattice points in a specific region.
In many cases, as we have seen in Section~\ref{sec_intro}, such a region will be described as the intersection of a thin annulus and a thin plate.
Thus, we prepare the following lemma.

\begin{lem}[Bound of lattice points]\label{lem_annulustrip}
Let $d\ge 2$, $N\gg 1$, $N^{-1}\le \mu ,\nu \ll N$, $X\ge 0$, and
\eqq{D:=\Shugo{\xi =(\xi _1,\xi ')\in \R \times \R ^{d-1}}{N\le |\xi |\le N+\mu ,\,X\le \xi _1\le X+\nu}.}
Then, denoting by $\Sc{R}$ an arbitrary rotation operator on $\R ^d$, we have the following estimates for any $1\le N_0\lec N$ and any ball  $B_{N_0}\subset \R ^d$ with radius $N_0$.

(i) $\# (\Zg ^d\cap \Sc{R}(D\cap B_{N_0})) \lec \max \shugo{\nu,1}N_0^{d-2}\big[ N(\mu +\min \shugo{\nu ,1})\big] ^{\frac{1}{2}}$. 

(ii) In addition, let $\xi _0:=(1,0,\dots ,0)\in \R^d$ and
\eqq{K_\th :=\Shugo{\xi \in \R^d}{\tfrac{\th}{2}\le \angle (\xi ,\xi _0)\le 2\th}}
for $(\frac{\mu +\min \shugo{\nu ,1}}{N})^{\frac{1}{2}}\ll \th \le \frac{\pi}{4}$.
Then, we have
\eqq{\# (\Zg ^d\cap \Sc{R}(D\cap B_{N_0}\cap K_\th ))\lec \max \shugo{\nu,1}\min \shugo{N\th ,N_0}^{d-2}\big[ \th ^{-1}(\mu +\min \shugo{\nu ,1})+1\big] .}
\end{lem}

\begin{proof}
We may restrict our attention to the case $\nu \le 1$; for $\nu >1$ the claim follows by slicing $D$ into $O(\nu)$ subdomains and applying the estimate for $\nu =1$ to each slice.

(i) Suppose $\xi =(\xi _1,\xi ')\in D\cap B_{N_0}$.
Consider first the case $X\ge N-(\mu +2\nu )$.
We observe that
\eqq{|\xi '|\le \sqrt{(N+\mu )^2-X^2}\le \sqrt{(N+\mu )^2-(N-\mu -2\nu )^2}\sim \sqrt{N(\mu +\nu )},}
 thus $\xi '$ is in a ball $\subset \R ^{d-1}$ of radius $\sim \min \shugo{\sqrt{N(\mu +\nu )},N_0}$.
Since the assumption $\mu ,\nu \ge N^{-1}$ implies $\sqrt{N(\mu +\nu )}\ge 1$, we see that $D\cap B_{N_0}$ is contained in the union of at most $O(\min \shugo{[N(\mu +\nu )]^{\frac{d-1}{2}},N_0^{d-1}})$ balls with radius $1$.
Therefore, we obtain the bound $\min \shugo{[N(\mu +\nu )]^{\frac{d-1}{2}},N_0^{d-1}}$, which is exceeded by the required bounds.

Next, assume $X\le N-(\mu +2\nu )$, which implies that
\eqq{\sqrt{N^2-(X+\nu )^2}\le |\xi '|\le \sqrt{(N+\mu )^2-X^2},}
and that $|\xi '|$ is contained to an interval of length
\eqq{\sqrt{(N+\mu )^2-X^2}-\sqrt{N^2-(X+\nu )^2}\lec \frac{N(\mu +\nu )}{\sqrt{N^2-(X+\nu )^2}}.}
Since $X+\nu \le N-(\mu +\nu )$, we have $\frac{N(\mu +\nu )}{\sqrt{N^2-(X+\nu )^2}}\lec \sqrt{N^2-(X+\nu )^2}$.
We also observe from the assumption that $\sqrt{N^2-(X+\nu )^2}\gec 1$.
Therefore, we obtain
\eqs{\Big( \frac{N(\mu +\nu )}{\sqrt{N^2-(X+\nu )^2}}+1\Big) \min \shugo{\big[ N^2-(X+\nu )^2\big] ^{\frac{d-2}{2}},\,N_0^{d-2}}\\
\lec \Big( \frac{N(\mu +\nu )}{\sqrt{N^2-(N-(\mu +\nu ))^2}}+1\Big) N_0^{d-2}\lec \big[ N(\mu +\nu )\big] ^{\frac{1}{2}}N_0^{d-2}}
as a bound.

(ii) Note that $\th \gg (\frac{\mu +\nu}{N})^{\frac{1}{2}}$ and $\mu ,\nu \ge N^{-1}$ imply $N\th \gg 1$.
If $\xi \in D\cap B_{N_0}\cap K_\th$, then
\eqq{\max \shugo{N\sin \frac{\th}{2},\,\sqrt{N^2-(X+\nu )^2}}\le |\xi '|\le \min \shugo{(N+\mu )\sin 2\th ,\,\sqrt{(N+\mu )^2-X^2}},}
so $\xi '\in \R ^{d-1}$ has to be in the intersection of a ball with radius $N_0$ and an annulus of radius $\sim N\th$ and thickness $\lec \frac{N(\mu +\nu )}{N\th}~(\ll N\th )$.
We thus obtain the claimed estimate.
\end{proof}

\section{Trilinear estimates}\label{sec_trilinear}

The required bilinear estimates are reduced to some trilinear estimates by duality.
In this section we shall prove various trilinear estimates for functions dyadically restricted both in the frequency variable $k$ and in the modulation variable $\tau +|k|^2$ or $\tau \pm |k|$.

\subsection{Estimate for the high-modulation interactions}
We begin with the trilinear estimate for the high-modulation cases, namely $L_{\max}\gec N_{\max}^2$.
As discussed later, this category includes the high-low interactions where one Schr\"odinger frequency is much greater than the other Schr\"odinger frequency.

\begin{prop}[High-modulation interactions]\label{prop_te_highlow}
Let $d\ge 2$.
Let $N_j,L_j\ge 1$ be dyadic numbers and $f,g_1,g_2\in L^2_{\ze}(\R \times \Zg ^d)$ be real-valued nonnegative functions with the support properties
\eqq{\Supp{f}{\FR{P}_{N_0}\cap \FR{W}_{L_0}^\pm},\quad \Supp{g_j}{\FR{P}_{N_j}\cap \FR{S}_{L_j}},\quad j=1,2.}
Assume $L_{\max}\gec N_{\max}^2$.
Then, we have
\eqq{\iint _{\ze _0=\ze _1-\ze _2}f(\ze _0)g_1(\ze _1)g_2(\ze _2)\lec L_{\max}^{\frac{1}{2}}L_{\med}^{\frac{1}{4}}L_{\min}^{\frac{1}{4}}N_{\min}^{\frac{d}{2}}N_{\max}^{-1}\norm{f}{L^2}\norm{g_1}{L^2}\norm{g_2}{L^2}.}
\end{prop}

\begin{proof}
An easy argument with the H\"older and the Young inequalities implies a bound of
\eqq{L_{\min}^{\frac{1}{2}}N_{\min}^{\frac{d}{2}}\norm{f}{L^2}\norm{g_1}{L^2}\norm{g_2}{L^2}\lec L_{\min}^{\frac{1}{2}}L_{\max}^{\frac{1}{2}}N_{\min}^{\frac{d}{2}}N_{\max}^{-1}\norm{f}{L^2}\norm{g_1}{L^2}\norm{g_2}{L^2},}
as required.
\end{proof}

Concerning the high-low interactions, the above argument will be sufficient for all the regularities satisfying \eqref{hani_3d} or \eqref{hani_2d} except for the border cases $l=\frac{d}{2}-1$ and $2s=l+\frac{d}{2}$, for which we can recover the estimates in the following way.
Note that a negative power of $N_{\min}$ will be obtained at the expense of more power in $L_j$'s.

\begin{prop}\label{prop_te_highlow2}
Let $d\ge 2$.
Let $f,g_1,g_2\in L^2_{\ze}(\R \times \Zg ^d)$ be functions as in Proposition~\ref{prop_te_highlow}, and assume $N_1\gg N_2$ or $N_2\gg N_1$.
Then, we have
\eqq{\iint _{\ze _0=\ze _1-\ze _2}f(\ze _0)g_1(\ze _1)g_2(\ze _2)\lec L_{\max}^{\frac{1}{2}}L_{\med}^{\frac{3}{8}}L_{\min}^{\frac{3}{8}}\underline{N}_{12}^{\frac{d-1}{2}}\overline{N}_{12}^{-1}\norm{f}{L^2}\norm{g_1}{L^2}\norm{g_2}{L^2}.}
\end{prop}

\begin{proof}
Without loss of generality we may assume $N_1\gg N_2$, which implies $N_0\sim N_1\gg 1$ whenever the integral is nonzero.
We can also assume $L_{\max}\gec N_{\max}^2$, because
\eqq{|\tau _0\pm |k_0||+|\tau _1+|k_1|^2|+|\tau _2+|k_2|^2|\ge |\pm |k_0|-|k_1|^2+|k_2|^2|\sim N_{\max}^2}
under the convention $\ze _0=\ze _1-\ze _2$.

We begin with the case $L_{\max}=L_0$.
Applying the Cauchy-Schwarz inequality in $\ze _0$ and Lemma~\ref{lem_bs} (i), we have
\eqq{\iint _{\ze _0=\ze _1-\ze _2}f(\ze _0)g_1(\ze _1)g_2(\ze _2)\lec \norm{f}{L^2}\underline{L}_{12}^{\frac{1}{2}}\Big( \frac{\overline{L}_{12}}{N_0}+1\Big) ^{\frac{1}{2}}N_{\min}^{\frac{d-1}{2}}\norm{g_1}{L^2}\norm{g_2}{L^2},}
with the prefactor bounded by
\eqq{\Big( L_{\min}^{\frac{1}{2}}\frac{L_{\med}^{\frac{1}{2}}}{\overline{N}_{12}^{\frac{1}{2}}}\cdot \frac{L_{\max}^{\frac{1}{4}}}{\overline{N}_{12}^{\frac{1}{2}}}+L_{\min}^{\frac{1}{2}}\cdot \frac{L_{\max}^{\frac{1}{2}}}{\overline{N}_{12}}\Big) \underline{N}_{12}^{\frac{d-1}{2}}\le L_{\min}^{\frac{3}{8}}L_{\med}^{\frac{3}{8}}L_{\max}^{\frac{1}{2}}\frac{\underline{N}_{12}^{\frac{d-1}{2}}}{\overline{N}_{12}},}
as desired.

For the case $L_{\max}=L_2$, we use the Cauchy-Schwarz in $\ze _2$ and Lemma~\ref{lem_bs} (ii) to obtain the bound
\eqq{\norm{g_2}{L^2}\underline{L}_{01}^{\frac{1}{2}}\Big( \frac{\overline{L}_{01}}{N_1}+1\Big) ^{\frac{1}{2}}N_{\min}^{\frac{d-1}{2}}\norm{f}{L^2}\norm{g_1}{L^2},}
which leads to an appropriate estimate in the same manner as above.

We finally treat the case $L_{\max}=L_1$, dividing the analysis into three subcases.

(a) $L_0\gec N_2^2$.
An application of the H\"older inequality in $\ze _0$ followed by the Young inequality for the convolution implies that
\eqq{\iint _{\ze _0=\ze _1-\ze _2}f(\ze _0)g_1(\ze _1)g_2(\ze _2)\le \norm{f}{\ell ^2_k(L^{8/5}_\tau )}\norm{g_1}{\ell ^2_k(L^2_\tau )}\norm{g_2}{\ell ^1_k(L^{8/7}_\tau )},}
which is, by the assumptions and the H\"older inequality again, estimated by
\begin{align}
&\lec L_0^{\frac{1}{8}}\norm{f}{L^2}\norm{g_1}{L^2}N_2^{\frac{d}{2}}L_2^{\frac{3}{8}}\norm{g_2}{L^2}\label{est_highlow1}\\
&\lec L_0^{\frac{1}{8}}\norm{f}{L^2}\norm{g_1}{L^2}N_2^{\frac{d}{2}}L_2^{\frac{3}{8}}\norm{g_2}{L^2}\cdot \frac{L_0^{\frac{1}{4}}}{N_2^{\frac{1}{2}}}\frac{L_{\max}^{\frac{1}{2}}}{\overline{N}_{12}},\notag
\end{align}
as desired.

(b) $L_2\gec N_2^2$.
This case is treated similarly to (a) if we apply the Young inequality as $\tnorm{f}{\ell ^2L^{8/7}}\tnorm{g_1}{\ell ^2L^2}\tnorm{g_2}{\ell ^1L^{8/5}}$.

(c) $\overline{L}_{02}\lec N_2^2$.
Applying the Cauchy-Schwarz in $\ze _1$ and then employing Lemma~\ref{lem_bs} (ii), we have an acceptable bound with prefactor
\eqq{\underline{L}_{02}^{\frac{1}{2}}\Big( \frac{\overline{L}_{02}}{N_2}+1\Big) ^{\frac{1}{2}}N_{\min}^{\frac{d-1}{2}}\lec \underline{L}_{02}^{\frac{1}{2}}\overline{L}_{02}^{\frac{1}{4}}\underline{N}_{12}^{\frac{d-1}{2}}\cdot \frac{L_{\max}^{\frac{1}{2}}}{\overline{N}_{12}}.\qedhere}
\end{proof}

In fact, we will use Corollary~\ref{cor_te_highlow} below for the high-low interactions.
The less power of $L_j$'s will lead to the longer local existence time of solutions, which will be important in constructing global solutions in the 2d case.
 
\begin{cor}[High-low interactions]\label{cor_te_highlow}
Under the same assumptions as in Proposition~\ref{prop_te_highlow2}, we have
\eqq{\iint _{\ze _0=\ze _1-\ze _2}f(\ze _0)g_1(\ze _1)g_2(\ze _2)\lec L_{\max}^{\frac{1}{2}}L_{\med}^{\frac{1}{4}+}L_{\min}^{\frac{1}{4}+}\underline{N}_{12}^{\frac{d}{2}-}\overline{N}_{12}^{-1}\norm{f}{L^2}\norm{g_1}{L^2}\norm{g_2}{L^2}.}
\end{cor}

\begin{proof}
An interpolation between Proposition~\ref{prop_te_highlow} and Proposition~\ref{prop_te_highlow2} shows the estimate.
\end{proof}

The interactions with very low wave frequency are also treated here.
Note that this case is a part of the high-high interactions to be discussed in the following two subsections.

\begin{cor}[Very low wave frequency]\label{cor_te_low}
Let $f,g_1,g_2\in L^2_{\ze}(\R \times \Zg ^d)$ be functions as in Proposition~\ref{prop_te_highlow}, and assume that $N_0\lec 1$.
Then, we have
\eqq{\iint _{\ze _0=\ze _1-\ze _2}f(\ze _0)g_1(\ze _1)g_2(\ze _2)\lec (L_0L_1L_2)^{\frac{1}{6}}\norm{f}{L^2}\norm{g_1}{L^2}\norm{g_2}{L^2}.}
\end{cor}

\begin{proof}
The first half of the proof of Proposition~\ref{prop_te_highlow} will be sufficient.
\end{proof}


\subsection{Estimate for the middle-modulation interactions}

We begin to establish the trilinear estimate for the high-high interactions in which two Schr\"odinger frequencies are comparable and not smaller than the wave frequency, namely $N_0\lec N_1\sim N_2$.
The case $N_0\lec 1$ is already finished in Corollary~\ref{cor_te_low}, and the case $L_{\max}\gec N_{\max}^2$ is treated with Proposition~\ref{prop_te_highlow}.
We now assume $N_0\gg 1$, and consider in this subsection the middle-modulation interactions, namely the case $N_{\max}\lec L_{\max}\ll N_{\max}^2$.

\begin{prop}[Middle-modulation high-high interactions]\label{prop_te_highhigh-m}
Let $d\ge 2$, and $f,g_1,g_2\in L^2_{\ze}(\R \times \Zg ^d)$ be functions as in Proposition~\ref{prop_te_highlow}.
Assume that $1\ll N_0\lec N_1\sim N_2\lec L_{\max}\ll N_1^2$.
Then, we have
\eqq{\iint _{\ze _0=\ze _1-\ze _2}f(\ze _0)g_1(\ze _1)g_2(\ze _2)\lec L_{\max}^{\frac{3}{8}+}L_{\med}^{\frac{3}{8}+}L_{\min}^{\frac{1}{4}}N_0^{\frac{d-2}{2}}\big( \frac{N_0}{N_1}\big) ^{0+}\norm{f}{L^2}\norm{g_1}{L^2}\norm{g_2}{L^2}.}
\end{prop}

\begin{proof}
\textbf{(I)} $L_{\max}=L_0$.
We consider two cases separately.

(a) If $\overline{L}_{12}\gec N_0$, we first apply the Cauchy-Schwarz inequality in $\ze _0$ and then Lemma~\ref{lem_bs} (i) to have a bound of
\eqq{\norm{f}{L^2}\underline{L}_{12}^{\frac{1}{2}}\overline{L}_{12}^{\frac{1}{2}}N_0^{-\frac{1}{2}}N_0^{\frac{d-1}{2}}\norm{g_1}{L^2}\norm{g_2}{L^2},}
which is sufficient after a multiplication by $\big( \frac{L_{\max}}{N_1}\big) ^{0+}\gec 1$.

(b) If $\overline{L}_{12}\lec N_0$, we take a different approach.
Observe that
\eqq{\big| |k_1|^2-|k_2|^2\big| =\big| \pm |k_0|-(\tau _0\pm |k_0|)+(\tau _1+|k_1|^2)-(\tau _2+|k_2|^2)\big| \lec N_0+L_{\max}}
in the integral domain, which implies $\big| |k_1|-|k_2|\big| \lec \frac{L_{\max}}{N_1}~(\ll N_1)$.
Therefore, by the orthogonality we can assume that $g_1$ and $g_2$ are localized (in $k$) to an annulus of radius $\sim N_1$ and thickness $\sim \frac{L_{\max}}{N_1}$ centered at the origin, as well as a ball of radius $\sim N_0$ (if $N_0\ll N_1$).
Also, we have seen in the proof of Lemma~\ref{lem_bs} (i) that $k_1$ satisfies
\eqq{\frac{k_0}{|k_0|}\cdot k_1=-\frac{\tau _0}{2|k_0|}+\frac{|k_0|}{2}+O(\frac{\overline{L}_{12}}{N_0}),}
thus belongs to a specific plate-like region of thickness $\sim \frac{\overline{L}_{12}}{N_0}~(\lec 1)$ for fixed $(\tau _0,k_0)$.
Now we apply Lemma~\ref{lem_annulustrip} (i) with $N\sim N_1$, $\mu \sim \frac{L_{\max}}{N_1}$, $\nu \sim \frac{\overline{L}_{12}}{N_0}$, and obtain that the integral is evaluated by
\eqq{\underline{L}_{12}^{\frac{1}{2}}\norm{f}{L^2}\norm{g_1}{L^2}\norm{g_2}{L^2}L_{\max}^{\frac{1}{4}}N_0^{\frac{d-2}{2}}.}
It then suffices to multiply it by $\big( \frac{L_{\max}}{N_1}\big) ^{0+}$.

\textbf{(II)} $L_{\max}=L_1$ or $L_2$.
Without loss of generality we assume $L_{\max}=L_2$.

(a) The case $\overline{L}_{01}\gec N_1$.
Applying Lemma~\ref{lem_bs} (ii) after the Cauchy-Schwarz in $\ze _2$, we have a sufficient bound of
\eqq{\norm{g_2}{L^2}\underline{L}_{01}^{\frac{1}{2}}\overline{L}_{01}^{\frac{1}{2}}N_1^{-\frac{1}{2}}N_0^{\frac{d-1}{2}}\norm{f}{L^2}\norm{g_1}{L^2}.}

(b) The case $\overline{L}_{01}\ll N_1\lec \frac{L_{\max}}{N_0}$.
Lemma~\ref{lem_bs} (ii) again implies a bound of
\eq{est_highhighm1}{\norm{g_2}{L^2}\underline{L}_{01}^{\frac{1}{2}}N_0^{\frac{d-1}{2}}\norm{f}{L^2}\norm{g_1}{L^2},}
which is not sufficient in general.
In the present case, however, we can multiply it by
\[ L_{\max}^{\frac{1}{4}}N_0^{-\frac{1}{4}}N_1^{-\frac{1}{4}}\gec 1\]
and obtain the claim.

(c) The case $\overline{L}_{01}\ll N_1$ and $L_{\max}\ll N_0N_1$.
We take the same approach as in the case \textbf{(I)}-(b).
Restricting $g_1$ and $g_2$ into an annulus of thickness $\sim \frac{L_{\max}}{N_1}$ and a ball of radius $\sim N_0$, we can apply Lemma~\ref{lem_annulustrip} (i) with $N\sim N_1$, $\mu \sim \frac{L_{\max}}{N_1}$ and $\nu \sim \frac{L_{\max}}{N_0}~(\ll N_1)$.
The bound is
\eq{est_highhighm2}{\norm{f}{L^2}L_1^{\frac{1}{2}}L_{\max}^{\frac{3}{4}}N_0^{\frac{d-3}{2}}\norm{g_1}{L^2}\norm{g_2}{L^2}.}
Finally, we take the $\frac{1}{2}$-interpolant between \eqref{est_highhighm1} and \eqref{est_highhighm2} and multiply it by $\big( \frac{L_{\max}}{N_1}\big) ^{0+}$ to obtain a suitable bound of
\eqq{L_0^{\frac{1}{4}}L_1^{\frac{1}{4}}L_2^{\frac{3}{8}+}N_0^{\frac{d-2}{2}}N_1^{0-}\norm{f}{L^2}\norm{g_1}{L^2}\norm{g_2}{L^2}.\qedhere}
\end{proof}


\subsection{Estimate for the low-modulation interactions}

We treat here the most dangerous case of low modulation, namely $1\ll N_0\lec N_1\sim N_2$ and $L_{\max}\ll N_1$.
As we have seen in Section~2, this case contains serious resonances which make it difficult to gain derivative (negative power of $N_1$).
We will need more careful case-by-case analysis including decomposition with respect to the angle between frequencies.

For any dimensions $d\ge 2$, it is possible to show some estimate yielding the control of the high-high interactions in the regularity range $2s\ge l+\frac{d}{2}>d-1$.
Moreover, if the spatial dimension is two, a little more consideration enables us to reach the border $l=0$, which includes the important regularity of the energy space.
Unfortunately, the same argument is not sufficient for the higher dimensional cases, and we leave the border case for $d\ge 3$ open.
However, some number theoretic method (\mbox{cf.} \cite{B07,dSPST}) might be applied to reach the border, and even lower regularities. 

Before the analysis, we recall that in the present case the following two identities are valid for $\ze _0$, $\ze _1$, $\ze _2$ in the integral region:
\eq{id-hh-1}{|k_1|-|k_2|&=\frac{1}{|k_1|+|k_2|}\big( \pm |k_0|-(\tau _0\pm |k_0|)+(\tau _1+|k_1|^2)-(\tau _2+|k_2|^2)\big) \\
&=\pm \frac{|k_0|}{|k_1|+|k_2|}+O(\frac{L_{\max}}{N_1})=O(\frac{N_0+L_{\max}}{N_1}),}
\eq{id-hh-2}{\frac{k_0}{|k_0|}\cdot k_1&=\pm \frac{1}{2}+\frac{|k_0|}{2}+\frac{1}{2|k_0|}\big( -(\tau _0\pm |k_0|)+(\tau _1+|k_1|^2)-(\tau _2+|k_2|^2)\big) \\
&=\pm \frac{1}{2}+\frac{|k_0|}{2}+O(\frac{L_{\max}}{N_0}).}

\begin{prop}[Low-modulation high-high interactions, $d\ge 3$]\label{prop_te_highhigh-l}
Let $d\ge 3$, and $f,g_1,g_2\in L^2_{\ze}(\R \times \Zg ^d)$ be functions as in Proposition~\ref{prop_te_highlow}.
Assume that $1\ll N_0\lec N_1\sim N_2$ and $L_{\max}\ll N_1$.
Then, we have
\eqq{\iint _{\ze _0=\ze _1-\ze _2}f(\ze _0)g_1(\ze _1)g_2(\ze _2)\lec L_{\max}^{\frac{3}{8}}L_{\med}^{\frac{3}{8}}N_0^{\frac{d-2}{2}}\norm{f}{L^2}\norm{g_1}{L^2}\norm{g_2}{L^2}.}
\end{prop}

\begin{proof}
We consider several cases separately.

\textbf{(I)} $1\ll N_0\lec L_{\max}$.
Taking \eqref{id-hh-1} into account, we can assume that $|k_1|$ and $|k_2|$ are restricted to an interval of length $\sim \frac{L_{\max}}{N_1}~(\ll 1)$.
The orthogonality also admits us to further localize $k_1$ and $k_2$ to a ball of size $\sim N_0$.
This and \eqref{id-hh-2} then say that the $\frac{k_0}{|k_0|}$-component of $k_1$ is confined to an interval of size $\sim \min \shugo{N_0,\frac{L_{\max}}{N_0}}$.
We also have $N_0\ll N_1$ and
\eqq{|\cos \angle (k_0,k_1) |=|\frac{|k_0|^2+|k_1|^2-|k_2|^2}{2|k_0||k_1|}|\lec \frac{L_{\max}+N_0^2}{N_0N_1}\ll 1.}
Therefore, we first apply the Cauchy-Schwarz in $\ze _0$ and then count the number of possible $k_1$'s for fixed $k_0$, which is, from Lemma~\ref{lem_annulustrip} (ii) with $\th \sim 1$, estimated by
\eqq{\min \shugo{N_0^{d-1},L_{\max}N_0^{d-3}}\lec L_{\max}^{\frac{1}{2}}N_0^{d-2},}
to obtain
\eqq{\iint _{\ze _0=\ze _1-\ze _2}f(\ze _0)g_1(\ze _1)g_2(\ze _2)\lec \underline{L}_{12}^{\frac{1}{2}}L_{\max}^{\frac{1}{4}}N_0^{\frac{d-2}{2}}\norm{f}{L^2}\norm{g_1}{L^2}\norm{g_2}{L^2}}
as required.

\textbf{(II)} $L_{\max}\ll N_0\ll N_1$.
Since we also have $|\cos \angle (k_0,k_1) |\ll 1$, this case is almost parallel to (I), except that we take $\nu \ll 1$ in applying Lemma~\ref{lem_annulustrip}.
The resulting bound is
\eqq{\underline{L}_{12}^{\frac{1}{2}}N_0^{\frac{d-2}{2}}\norm{f}{L^2}\norm{g_1}{L^2}\norm{g_2}{L^2}.}

\begin{rem}
In the case of (II), it seems nontrivial to show a similar estimate but with prefactor $N_0^{\frac{d-2}{2}-}$ or $N_0^{\frac{d-2}{2}}(\frac{N_0}{N_1})^{0+}$, even if we pay any amount of $L_{\max}$.
This is exactly the reason why the border case is left open for $d\ge 3$.
\end{rem}

\textbf{(III)} $N_0\sim N_1$.
In this case we have $\cos \angle (k_0,k_1)\sim \cos \angle (-k_0,k_2)\sim +1$, hence $\pi \ge \angle (k_1,k_2)\gec 1$.
For the region where $\angle (k_1,k_2)\le \frac{\pi}{2}$, we just recall \eqref{id-hh-1} and \eqref{id-hh-2} and apply Lemma~\ref{lem_annulustrip}~(ii) with $\th \sim 1$, which implies a suitable bound of
\eqq{\underline{L}_{12}^{\frac{1}{2}}N_0^{\frac{d-2}{2}}\norm{f}{L^2}\norm{g_1}{L^2}\norm{g_2}{L^2}.}

Let us next deal with another easy case of $\angle (k_1,-k_2)\lec L_{\max}^{\frac{1}{2}}N_1^{-1}$.
Having restricted $|k_1|$ and $|k_2|$ to an annulus, we may further restrict them into a ball of radius $\sim L_{\max}^{\frac{1}{2}}$.
Then, the number of $k_1$'s for fixed $k_0$ is bounded by $L_{\max}^{\frac{d-1}{2}}\lec L_{\max}^{\frac{1}{2}}N_0^{\frac{d-2}{2}}$, and thus
\eqq{\iint _{\ze _0=\ze _1-\ze _2}f(\ze _0)g_1(\ze _1)g_2(\ze _2)\lec \underline{L}_{12}^{\frac{1}{2}}L_{\max}^{\frac{1}{4}}N_0^{\frac{d-2}{4}}\norm{f}{L^2}\norm{g_1}{L^2}\norm{g_2}{L^2}.}

For the remaining cases, namely $L_{\max}^{\frac{1}{2}}N_1^{-1}\ll \angle (k_1,-k_2)\le \frac{\pi}{2}$, we treat separately each of the integral in the region $\angle (k_1,-k_2)\sim \phi$ for dyadic $L_{\max}^{\frac{1}{2}}N_1^{-1}\ll \phi \le 1$.
Since we have
\eqq{1-\Big( \frac{|k_0|}{|k_1|+|k_2|}\Big) ^2=\frac{2|k_1||k_2|}{(|k_1|+|k_2|)^2}\big( 1+\cos \angle (k_1,k_2)\big) \sim \phi ^2,}
\eqref{id-hh-1} actually says that $||k_1|-|k_2||=1+O(\phi ^2+\frac{L_{\max}}{N_1})$.
We also observe that
\eqq{(\frac{\pi}{2}\ge )~\angle (k_0,k_1)\sim \sin \angle (k_0,k_1)=\frac{|k_2|}{|k_0|}\sin \angle (k_1,k_2)\sim \phi.}
We can thus apply Lemma~\ref{lem_annulustrip} (ii) with $N\sim N_1\sim N_0$, $\mu \sim \phi ^2+\frac{L_{\max}}{N_0}$, $\nu \sim \frac{L_{\max}}{N_0}$, and $\th \sim \phi$.
Note that the condition $\th \gg (\frac{\mu +\min \shugo{\nu ,1}}{N})^{\frac{1}{2}}$ is satisfied if $\phi \gg L_{\max}^{\frac{1}{2}}N_1^{-1}$.
We finally obtain the following bounds of the number of $k_1$'s for fixed $k_0$:
\eqq{(N_0\phi )^{d-2}(\phi ^{-1}\cdot \phi ^2 +1)\sim (N_0\phi )^{d-2}}
for $\phi \gg L_{\max}^{\frac{1}{2}}N_0^{-\frac{1}{2}}$,
\eqq{(N_0\phi )^{d-2}(\phi ^{-1}\frac{L_{\max}}{N_0} +1)\sim (N_0\phi )^{d-2}}
for $L_{\max}^{\frac{1}{2}}N_0^{-\frac{1}{2}}\gec \phi \gg L_{\max}N_0^{-1}$, and
\eqq{(N_0\phi )^{d-2}(\phi ^{-1}\frac{L_{\max}}{N_0} +1)\sim (N_0\phi )^{d-2}\phi ^{-1}\frac{L_{\max}}{N_0}\lec L_{\max}^{\frac{1}{2}}(N_0\phi )^{d-2}}
for $L_{\max}N_0^{-1}\gec \phi \gg L_{\max}^{\frac{1}{2}}N_0^{-1}$, which imply the corresponding bound of $\iint fg_1g_2$ for each $\phi$.
It is then sufficient to sum up these estimates over dyadic $\phi \le 1$.
\end{proof}

\begin{prop}[Low-modulation high-high interactions, $d=2$]\label{prop_te_highhigh-l_2d}
Let $d=2$.
We do not decompose $f$ in $N_0$, and let $f,g_1,g_2\in L^2_{\ze}(\R \times \Zg ^d)$ be real-valued nonnegative functions with the support properties
\eqq{\Supp{f}{\shugo{|k|\gg 1}\cap \FR{W}_{L_0}^\pm},\quad \Supp{g_j}{\FR{P}_{N_j}\cap \FR{S}_{L_j}},\quad j=1,2.}
Assume that $1\ll N_1\sim N_2$ and $L_{\max}\ll N_1$.
Then, we have
\eqq{\iint _{\ze _0=\ze _1-\ze _2}f(\ze _0)g_1(\ze _1)g_2(\ze _2)\lec L_{\max}^{\frac{3}{8}}L_{\med}^{\frac{3}{8}}\norm{f}{L^2}\norm{g_1}{L^2}\norm{g_2}{L^2}.}
\end{prop}

\begin{proof}
We follow the proof of the previous proposition.

\textbf{(I)} The case $1\ll |k_0|\lec L_{\max}$.
In this case we temporarily decompose $f$ in $N_0$ for the estimate.
Applying Lemma~\ref{lem_annulustrip} (ii) with $N\sim N_1$, $\mu \sim \frac{L_{\max}}{N_1}$, $\nu \sim \min \shugo{ N_0,\frac{L_{\max}}{N_0}}$, and $\th \sim 1$, we have a bound of
\eqq{\underline{L}_{12}^{\frac{1}{2}}\min \shugo{ N_0^{\frac{1}{2}},\,(\frac{L_{\max}}{N_0})^{\frac{1}{2}}}\norm{f}{L^2}\norm{g_1}{L^2}\norm{g_2}{L^2}.}
Summing this over dyadic $N_0$, we obtain the desired estimate.

\textbf{(II)} The case $L_{\max}\ll |k_0|\ll N_1$.
This time we can employ Lemma~\ref{lem_annulustrip} (ii) with $\nu \sim 1$.
The resulting bound is
\eqq{\underline{L}_{12}^{\frac{1}{2}}\norm{f}{L^2}\norm{g_1}{L^2}\norm{g_2}{L^2}.}

\textbf{(III)} $|k_0|\sim N_1$.
The two cases of $\angle (k_1,k_2)\le \frac{\pi}{2}$ and $\angle (k_1,-k_2)\lec L_{\max}^{\frac{1}{2}}N_1^{-1}$ are treated exactly in the same way as for $d\ge 3$.
For the region $L_{\max}^{\frac{1}{2}}N_1^{-1}\ll \angle (k_1,-k_2)\lec L_{\max}N_1^{-1}$, we divide dyadically with respect to $\angle (k_1,-k_2)$ and make the same argument as for $d\ge 3$.
The resulting bound for $\angle (k_1,-k_2)\sim \phi$ is
\eqq{\underline{L}_{12}^{\frac{1}{2}}\big( \phi ^{-1}\frac{L_{\max}}{N_1}\big) ^{\frac{1}{2}}\norm{f}{L^2}\norm{g_1}{L^2}\norm{g_2}{L^2},}
which can be summed over $\phi \gg L_{\max}^{\frac{1}{2}}N_1^{-1}$ to yield the claimed estimate.
For the middle angle $L_{\max}N_1^{-1}\ll \angle (k_1,-k_2)\lec L_{\max}^{\frac{1}{2}}N_1^{-\frac{1}{2}}$, we do not decompose with respect to the angle and use Lemma~\ref{lem_annulustrip} (ii) directly.
Note that if $k_0$ is fixed, $k_1$ is confined to the intersection of specific annulus and band, in which $\angle (k_0,k_1)$ does not vary so much.
Therefore, for each $k_0$, we can apply Lemma~\ref{lem_annulustrip} (ii) with some single value of dyadic $\th$ between $L_{\max}N_1^{-1}$ and $L_{\max}^{\frac{1}{2}}N_1^{-\frac{1}{2}}$.
The result is 
\eqq{\iint _{\ze _0=\ze _1-\ze _2}f(\ze _0)g_1(\ze _1)g_2(\ze _2)\lec \underline{L}_{12}^{\frac{1}{2}}\norm{f}{L^2}\norm{g_1}{L^2}\norm{g_2}{L^2}.}

Only the case $L_{\max}^{\frac{1}{2}}N_1^{-\frac{1}{2}}\ll \angle (k_1,-k_2)\le \frac{\pi}{2}$ is troublesome.
In order to avoid a logarithmic divergence (\mbox{i.e.} estimate with $\log N_1$), we decompose all of $k_0$, $k_1$, and $k_2$ as follows.
First, taking \eqref{id-hh-1} into account, restrict $k_1$ and $k_2$ to a common annulus $\shugo{N\le |k|\le N+10}$, where $N\sim N_1$.
Then we make an angular decomposition of angular aperture $\sim \sqrt{L_{\max}/N_1}$; define
\eqq{D_j:=\Shugo{(r\cos \th ,r\sin \th )\in \R ^2}{N\le r\le N+10,\,2\pi (j-1)\le \th \sqrt{\tfrac{N_1}{L_{\max}}}\le 2\pi (j+1)}}
for $j=0,1,\dots ,\sqrt{N_1/L_{\max}}-1$, and localize each of $k_1$ and $k_2$ to one of them.
From the assumption on $\angle (k_1,-k_2)$, we only need to consider $k_1\in D_{j_1}$, $k_2\in D_{j_2}$ with
\eqq{(j_1,j_2)\in \Sc{J}:=\Big\{ 0\le j_1,j_2\le \sqrt{\tfrac{N_1}{L_{\max}}}-1,\; \tfrac{1}{4}\sqrt{\tfrac{N_1}{L_{\max}}}-2\le d[j_1,j_2]\le \tfrac{1}{2}\sqrt{\tfrac{N_1}{L_{\max}}}-100\Big\},}
where $d[j_1,j_2]:=\min \shugo{|j_1-j_2|,\,\sqrt{N_1/L_{\max}}-|j_1-j_2|}$.
Also, we localize $k_0$ to a similar region
\eqq{\ti{D}_{j_r,j_\th }&:=\big\{ (r\cos \th ,r\sin \th )\in \R ^2\big| \;2\pi (j_\th -1)\le \th \sqrt{\tfrac{N_1}{L_{\max}}}\le 2\pi (j_\th +1),\\
&\hspace{80pt} 2N\sin \Big[ \pi (j_r -1)\sqrt{\tfrac{L_{\max}}{N_1}}\Big] \le r\le 2N\sin \Big[ \pi (j_r +1)\sqrt{\tfrac{L_{\max}}{N_1}}\Big] \big\} }
for some
\eqq{(j_r,j_\th )\in \ti{\Sc{J}}:=\shugo{1,2,\dots ,\tfrac{1}{2}\sqrt{\tfrac{N_1}{L_{\max}}}-1}\times \shugo{0,1,\dots, \sqrt{\tfrac{N_1}{L_{\max}}}-1}.}

The following is the key orthogonality lemma.
\begin{lem}[Orthogonality]\label{lem_orthogonality}
Assume that $N\sim N_1\gg 1$, $L_{\max}\ge 1$, and that $N_1/L_{\max}$ is sufficiently large.
Then, there is a two-to-one mapping $\kappa =(\kappa _r,\kappa _\th ):\Sc{J}\to \ti{\Sc{J}}$ such that 
\eqq{\Shugo{k_1-k_2}{k_1\in D_{j_1},\,k_2\in D_{j_2}}\subset \bigcup _{(j_r,j_\th )\in B(j_1,j_2)}\ti{D}_{j_r,j_\th}}
for any $(j_1,j_2)\in \Sc{J}$, where
\eqq{B(j_1,j_2):=\Shugo{(j_r,j_\th )\in \ti{\Sc{J}}}{|\kappa _r(j_1,j_2)-j_r|\le 10,\,d[\kappa _\th (j_1,j_2),j_\th ]\le C}}
with a large constant $C>0$.
\end{lem}

\begin{proof}
Let us first define $\kappa$ (see Figure~2).
\begin{figure}
\begin{center}
\input{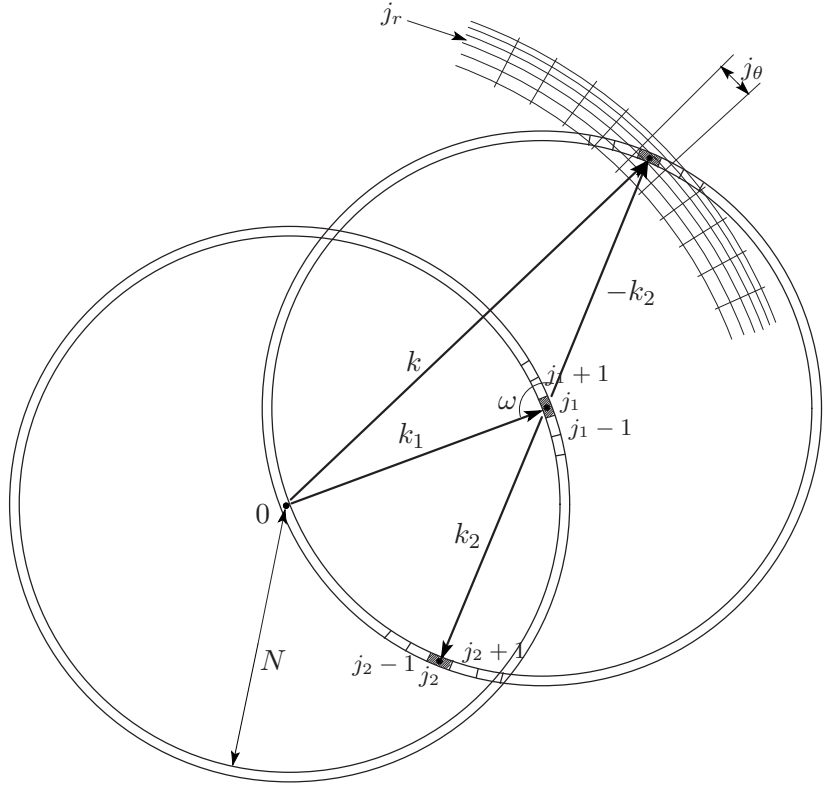}
\end{center}
\label{fig_orthogonality}
\caption{Correspondence between $(j_1,j_2)$ and $(j_r,j_\th)$ in Lemma~\ref{lem_orthogonality}.
With $N$ fixed, $j_r$ (length of $k=k_1-k_2$) is (essentially uniquely) determined by $d[j_1,j_2]$ (angle between $k_1$ and $k_2$, denoted in the figure by $\omega$).
Then, with $d[j_1,j_2]$ fixed, (essentially) two candidates for $j_\th$ (angle of $k$) are determined by $j_1$.}
\end{figure}
As the representative element of $D_j$, let
\[ k_*(j):=(N\cos 2\pi j\sqrt{\tfrac{L_{\max}}{N_1}},N\sin 2\pi j\sqrt{\tfrac{L_{\max}}{N_1}}).\]
We see that
\eqq{|k_*(j_1)-k_*(j_2)|=2N\sin \Big[ \pi d[j_1,j_2]\sqrt{\tfrac{L_{\max}}{N_1}}\Big] ,}
so define $\kappa _r(j_1,j_2):=d[j_1,j_2]$.
For $\kappa _\th$, define first $\kappa _\th (j_1,0)$ as an arbitrary integer in $[0,\sqrt{\frac{L_{\max}}{N_1}}-1]$ such that $k_*(j_1)-k_*(0)\in \ti{D}_{d[j_1,0],\kappa _\th (j_1,0)}$, and then
\eqq{\kappa _\th (j_1,j_2):=j_2+ \begin{cases} \kappa _\th (j_1-j_2,0) &(j_1\ge j_2),\\ \kappa _\th (j_1-j_2+\sqrt{\tfrac{N_1}{L_{\max}}},0) &(j_1<j_2),\end{cases}}
for $j_2\neq 0$.
If $\kappa _\th \ge \sqrt{\tfrac{N_1}{L_{\max}}}$, we re-define it by $\kappa _\th -\sqrt{\tfrac{N_1}{L_{\max}}}$.
Note that $k_*(j_1)-k_*(j_2)\in \ti{D}_{\kappa (j_1,j_2)}$.

It is clear from the definition that there are exactly two elements in $\Sc{J}$, or nothing, which satisfy $\kappa (j_1,j_2)=(j_r,j_\th )$ for a specific $(j_r,j_\th )\in \ti{\Sc{J}}$.

Let $k_1$ and $k_2$ be arbitrary elements in $D_{j_1}$ and $D_{j_2}$, respectively.
We observe that
\eqq{2N\sin \Big[ \pi (d[j_1,j_2]-2)\sqrt{\tfrac{L_{\max}}{N_1}}\Big] \le |k_1-k_2|\le 2(N+10)\sin \Big[ \pi (d[j_1,j_2]+2)\sqrt{\tfrac{L_{\max}}{N_1}}\Big] .}
Under the assumption that $d[j_1,j_2]$ ranges between $\frac{1}{4}\sqrt{\tfrac{N_1}{L_{\max}}}-2$ and $\tfrac{1}{2}\sqrt{\tfrac{N_1}{L_{\max}}}-100$, it is easily verified that
\eqq{
2(N+10)\sin \Big[ \pi (d[j_1,j_2]+2)\sqrt{\tfrac{L_{\max}}{N_1}}\Big] \le 2N\sin \Big[ \pi (d[j_1,j_2]+10)\sqrt{\tfrac{L_{\max}}{N_1}}\Big] .}
Moreover, we see that
\eqq{|(k_1-k_2)-(k_*(j_1)-k_*(j_2))|\le \mathrm{diam} D_{j_1}+\mathrm{diam} D_{j_2}\sim N_1\sqrt{\tfrac{L_{\max}}{N_1}},}
thus the angle between $k_1-k_2$ and $k_*(j_1)-k_*(j_2)$ is $O(\sqrt{\frac{L_{\max}}{N_1}})$.
The claim follows from these facts.
\end{proof}

Let us go back to the proof of Proposition~\ref{prop_te_highhigh-l_2d}.
Thanks to the orthogonality lemma, we are allowed to focus on a situation that $k_0,k_1,k_2$ are localized to some specific $\ti{D}_{j_r,j_\th}$, $D_{j_1}$, and $D_{j_2}$ respectively.
In fact, an application of the Cauchy-Schwarz inequality in $j_1,j_2$ yields the desired estimate from decomposed estimates with respect to $j_1$, $j_2$, and $(j_r,j_\th)$.
Under this localization, the angle between $k_1$ and $-k_2$ is comparable to some $\sqrt{\frac{L_{\max}}{N_1}}\ll \phi \le 1$.
Therefore, we restrict $|k_1|$ and $|k_2|$ further onto some intervals of length $\sim \phi ^2$ and follow the argument for the case $d\ge 3$, obtaining the estimate
\eqq{\iint _{\ze _0=\ze _1-\ze _2}f(\ze _0)g_1(\ze _1)g_2(\ze _2)\lec L_{\med}^{\frac{1}{2}}\norm{f}{L^2}\norm{g_1}{L^2}\norm{g_2}{L^2}}
for this case.
\end{proof}


\section{Proof of Theorem~\ref{thm_main}}\label{sec_pf-main}

In this section we apply the contraction mapping argument to the integral equations \eqref{IE} to prove the local well-posedness, Theorem~\ref{thm_main3}.
Define the Duhamel operators
\eqs{\Sc{I}_SF(t):=-i\int _0^te^{i(t-t')\Delta}F(t')\,dt',\qquad \Sc{I}_{W_{\pm}}G(t):=i\int _0^te^{\mp i(t-t')\LR{\nabla}}G(t')\,dt'.}

The following linear estimates will be used.
Positive power of $\de$ included in these estimates enables us to deal with data with arbitrary size and verify the uniqueness of solutions in the Bourgain spaces.
We will use a bump function $\psi _\de(t):=\psi (t/\de )$, where $\psi \in C^\I_0(\R)$ is a function with the same property as $\eta$ given in Definition~\ref{def_LPdec}.

\begin{lem}[Linear estimates]\label{lem_linear}
Let $s\in \R$.
For any $0<\de \le 1$ and $0<b\le \frac{1}{2}$, the following estimates hold.
The implicit constants do not depend on $s$, $\de$.
\begin{gather}
\norm{\psi _\de e^{it\Delta}u_0}{X_S^{s,\frac{1}{2},1}}\lec \norm{u_0}{H^s},\qquad \norm{\psi _\de e^{-it\LR{\nabla}}w_0}{X_{W_+}^{s,\frac{1}{2},1}}\lec \norm{w_0}{H^s},\label{est_lin_sol}\\
\norm{\psi _\de u}{X^{s,b,1}_S}\lec \de ^{\frac{1}{2}-b}\norm{u}{X^{s,\frac{1}{2},1}_S},\qquad \norm{\psi _\de w}{X^{s,b,1}_{W_{\pm}}}\lec \de ^{\frac{1}{2}-b}\norm{w}{X^{s,\frac{1}{2},1}_{W_{\pm}}},\label{est_stability}\\
\norm{\psi _\de \Sc{I}_SF}{X^{s,\frac{1}{2},1}_S}\lec \de ^{\frac{1}{2}-b}\norm{F}{X^{s,-b,1}_S},\qquad \norm{\psi _\de \Sc{I}_{W_{\pm}}G}{X^{s,\frac{1}{2},1}_{W_{\pm}}}\lec \de ^{\frac{1}{2}-b}\norm{G}{X^{s,-b,1}_{W_{\pm}}}.\label{est_lin_duh}
\end{gather}
\end{lem}
\begin{rem}\label{rem_wavemoissho}
We also have
\eqq{\norm{\psi _\de \int _0^te^{\mp i(t-t')|\nabla |}G(t')\,dt'}{X^{s,-\frac{1}{2},1}_{W_{\pm}}}\lec \de ^{\frac{1}{2}-b}\norm{G}{X^{s,-b,1}_{W_{\pm}}}}
for $0<b\le \frac{1}{2}$ by the same proof.
\end{rem}

\begin{proof}
We shall consider only the estimates for the Schr\"odinger case and write $X^{s,b,1}$ to denote $X^{s,b,1}_S$.
Once we note that $\LR{\tau \pm |k|}\sim \LR{\tau \pm \LR{k}}$, proof for the wave case will be identical.
Although most of these estimates were proved in \cite{BHHT}, we give a complete proof.

Before verifying the claim, we observe that
\eq{est_psi_besov}{\norm{\psi _\de}{B^b_{2,1}}\lec \de ^{\frac{1}{2}-b}\norm{\psi}{B^b_{2,1}}\lec \de ^{\frac{1}{2}-b}}
for $0<\de \le 1$ and $b>0$.
In fact, we have
\eqq{\norm{\psi _\de}{B^b_{2,1}}&\sim \sum _{L\ge 1}L^b\norm{\eta _L\de \hhat{\psi}(\de \cdot )}{L^2(\R )}\\
&\sim \de ^{\frac{1}{2}}\sum _{1\le L\le \de ^{-1}}L^b\norm{\eta _L(\de ^{-1}\cdot )\hhat{\psi}}{L^2}+\de ^{\frac{1}{2}}\sum _{L>\de ^{-1}}\de ^{-b}(\de L)^b\norm{\eta _{\de L}\hhat{\psi}}{L^2}\\
&\lec \de ^{\frac{1}{2}-b}(\norm{\hhat{\psi}}{L^2(|\tau |\lec 1)}+\norm{\psi}{B^b_{2,1}}).}

By the definition of the $X^{s,b,1}$ norm, we see that
\eqq{\norm{\psi _\de e^{it\Delta}u_0}{X^{s,\frac{1}{2},1}}=\norm{\psi _\de}{B^{\frac{1}{2}}_{2,1}}\norm{u_0}{H^s}\lec \norm{u_0}{H^s},}
which shows \eqref{est_lin_sol}.

We now estimate
\eqq{\norm{\psi _\de u}{X^{s,b,1}}=\Big[ \sum _{N\ge 1}N^{2s}\Big( \sum _{L\ge 1}L^b\norm{\eta _N(k)\eta _L(\tau +|k|^2)\int _\R \hhat{\psi _\de}(\tau -\tau ')\ti{u}(\tau ',k)\,d\tau '}{L^2_\tau \ell ^2_k}\Big) ^2\Big] ^{\frac{1}{2}}.}
Combining the Young and the H\"older inequalities, we obtain, for fixed $k$ and any dyadic $L,L_1,L_2\ge 1$, that
\eqq{&\norm{\eta _L(\tau +|k|^2)\int _\R (\eta _{L_1}\hhat{\psi _\de})(\tau -\tau ')\eta _{L_2}(\tau '+|k|^2)\ti{u}(\tau ',k)\,d\tau '}{L^2_\tau}\\
&\lec \min \shugo{L_1,L_2}^b\norm{\eta _{L_1}\hhat{\psi _\de}}{L^2}\norm{\eta _{L_2}(\cdot +|k|^2)\ti{u}(\cdot ,k)}{L^{\frac{1}{1-b}}}\\
&\lec \min \shugo{L_1,L_2}^bL_2^{\frac{1}{2}-b}\norm{\eta _{L_1}\hhat{\psi _\de}}{L^2}\norm{\eta _{L_2}(\cdot +|k|^2)\ti{u}(\cdot ,k)}{L^2}}
for $0\le b\le \frac{1}{2}$.
Since the above norm vanishes unless $L\lec \max \shugo{L_1,L_2}$, we have
\eqq{&\sum _{L,L_1,L_2\ge 1}L^b\norm{\eta _N(k)\eta _L(\tau +|k|^2)\int _\R (\eta _{L_1}\hhat{\psi _\de})(\tau -\tau ')\eta _{L_2}(\tau '+|k|^2)\ti{u}(\tau ',k)\,d\tau '}{L^2_\tau \ell ^2_k}\\
&\lec \sum _{L_1,L_2\ge 1}L_1^bL_2^{\frac{1}{2}}\norm{\eta _N(k)\tnorm{\eta _{L_1}\hhat{\psi _\de}}{L^2}\tnorm{\eta _{L_2}(\cdot +|k|^2)\ti{u}(\cdot ,k)}{L^2_\tau }}{\ell ^2_k}}
for $0<b\le \frac{1}{2}$.
Then, \eqref{est_stability} follows from \eqref{est_psi_besov}.

We next treat
\begin{align}
\psi _\de \Sc{I}_SF&=\psi _\de (t)\F _k^{-1}e^{-it|k|^2}\int _\R \ti{F}(\tau ,k)\int _0^te^{it'(\tau +|k|^2)}\,dt'\,d\tau \notag \\
&=\psi _\de (t)\F _k^{-1}e^{-it|k|^2}\int _\R \ti{F}(\tau ,k)\psi _{\de ^{-1}}(\tau +|k|^2)\sum _{n=1}^\I \frac{t^n}{n!}\big[ i(\tau +|k|^2)\big] ^{n-1}\,d\tau \label{term_duhamel1}\\
&\hx -\psi _\de (t)\F _k^{-1}e^{-it|k|^2}\int _\R \ti{F}(\tau ,k)\frac{1-\psi _{\de ^{-1}}(\tau +|k|^2)}{i(\tau +|k|^2)}d\tau \label{term_duhamel2}\\
&\hx +\psi _\de (t)\F _k^{-1}\int _\R e^{it\tau }\ti{F}(\tau ,k)\frac{1-\psi _{\de ^{-1}}(\tau +|k|^2)}{i(\tau +|k|^2)}d\tau .\label{term_duhamel3}
\end{align}
From \eqref{est_psi_besov} we see that
\eqq{\norm{t^n\psi _\de}{B^{\frac{1}{2}}_{2,1}}=\de ^n\norm{(t^n\psi )_\de}{B^{\frac{1}{2}}_{2,1}}\lec \de ^n\big( \norm{t^n\psi}{L^2}+\norm{\p _t(t^n\psi )}{L^2}\big) \lec (C\de )^n,}
and thus
\eqq{&\norm{\eqref{term_duhamel1}}{X^{s,\frac{1}{2},1}}\lec \sum _{n=1}^\I \frac{(C\de )^n}{n!}\norm{\LR{k}^s\int _\R |\ti{F}(\tau ,k)|\frac{\de ^{-(n-\frac{1}{2}+b)}}{\LR{\tau +|k|^2}^{\frac{1}{2}+b}}\,d\tau }{\ell ^2_k}}
for $b\ge -\frac{1}{2}$.
Decomposing dyadically and applying the Cauchy-Schwarz inequality in $\tau$, we evaluate the above by $\de ^{\frac{1}{2}-b}\tnorm{F}{X^{s,-b,1}}$.
Similarly, we use \eqref{est_lin_sol} (\mbox{resp.} \eqref{est_stability}) to estimate the $X^{s,\frac{1}{2},1}$ norm of \eqref{term_duhamel2} (\mbox{resp.} \eqref{term_duhamel3}) by $\de ^{\frac{1}{2}-b}\tnorm{F}{X^{s,-b,1}}$, for $b\le \frac{1}{2}$.
\end{proof}

Finally, we combine all the trilinear estimates proved in the preceding section and some of the above linear estimates to establish the crucial bilinear estimates.
\begin{prop}[Bilinear estimates]\label{prop_be}
Let $d\ge 2$ and $(s,l)$ satisfy \eqref{hani_3d} or \eqref{hani_2d}.
Then, we have
\begin{align}
\norm{\Sc{I}_S(uw)}{X^{s,\frac{1}{2},1}_S(\de )}+\norm{\Sc{I}_S(u\bar{w})}{X^{s,\frac{1}{2},1}_S(\de )}&\lec \de ^{\frac{1}{2}-}\norm{u}{X^{s,\frac{1}{2},1}_S(\de )}\norm{w}{X^{l,\frac{1}{2},1}_{W_{\pm}}(\de )},\label{est_be_s}\\
\norm{\Sc{I}_{W_{\pm}}(\tfrac{\Delta _\be}{\LR{\nabla}}(u\bar{v}))}{X^{l,\frac{1}{2},1}_{W_{\pm}}(\de )}&\lec \de ^{\frac{1}{2}-}\norm{u}{X^{s,\frac{1}{2},1}_S(\de )}\norm{v}{X^{s,\frac{1}{2},1}_S(\de )}\label{est_be_w}
\end{align}
for $0<\de \le 1$.
\end{prop}
\begin{rem}\label{rem_wavemoissho2}
In view of Remark~\ref{rem_wavemoissho}, we also have the estimate \eqref{est_be_w} with the reduced-wave Duhamel operator $\mathcal{I}_{W_{\pm}}G$ replaced by $\int _0^te^{\mp i(t-t')|\nabla |}G(t')\,dt'$.
\end{rem}

\begin{proof}
For \eqref{est_be_s}, we note the relation $\tnorm{\bar{w}}{X^{s,b,1}_{W_{\pm}}}=\tnorm{w}{X^{s,b,1}_{W_{\mp}}}$ to restrict our attention to the estimate of the first term.

We make the Littlewood-Paley decompositions,
\eqq{&\norm{\Sc{I}_S(uw)}{X^{s,\frac{1}{2},1}_S(\de )}\le \norm{\psi _\de \Sc{I}_S(\psi _\de u\cdot \psi_\de w)}{X^{s,\frac{1}{2},1}_S}\\
&\sim \Big[ \sum _{N_1\ge 1}\norm{P_{N_1}\psi _\de \Sc{I}_S(\psi _\de u\cdot \psi_\de w)}{X^{s,\frac{1}{2},1}_S}^2\Big] ^{\frac{1}{2}}\le \Big[ \sum _{N_1\ge 1}\Big( \sum _{N_0,N_2\ge 1}\sum _{L_0,L_1,L_2\ge 1}\Sc{N}^S\Big) ^2\Big] ^{\frac{1}{2}},\\
&\Sc{N}^S=\Sc{N}^S_{N_0,N_1,N_2,L_0,L_1,L_2}:=\norm{\psi _\de \Sc{I}_S(P^S_{N_1,L_1}[P^S_{N_2,L_2}(\psi _\de u)\cdot P^{W_{\pm}}_{N_0,L_0}(\psi_\de w)])}{X^{s,\frac{1}{2},1}_S},}
then it suffices to evaluate the above by $\de ^{\frac{1}{2}-}\tnorm{u}{X^{s,\frac{1}{2},1}_S}\tnorm{w}{X^{l,\frac{1}{2},1}_{W_{\pm}}}$.
We will apply Lemma~\ref{lem_linear} and the trilinear estimates in the preceding sections, considering the following nine disjoint cases:
\eqq{(0)\quad&N_0\lec 1,\\
(1a)\quad&N_0\gg 1,\quad N_1\gg N_2,\quad \overline{L}_{02}\gec N_0^2,\\
(1b)\quad&N_0\gg 1,\quad N_1\gg N_2,\quad \overline{L}_{02}\ll N_0^2,\\
(2a)\quad&N_0\gg 1,\quad N_2\gg N_1,\quad \overline{L}_{02}\gec N_0^2,\\
(2b)\quad&N_0\gg 1,\quad N_2\gg N_1,\quad \overline{L}_{02}\ll N_0^2,\\
(3a)\quad&N_0\gg 1,\quad N_1\sim N_2,\quad \overline{L}_{02}\gec N_1^2,\\
(3b)\quad&N_0\gg 1,\quad N_1\sim N_2,\quad \overline{L}_{02}\ll N_1^2,\quad L_1\gec N_1^2,\\
(4a)\quad&N_0\gg 1,\quad N_1\sim N_2,\quad \overline{L}_{02}\ll N_1^2,\quad \underline{L}_{02}\lec L_1\ll N_1^2,\\
(4b)\quad&N_0\gg 1,\quad N_1\sim N_2,\quad \overline{L}_{02}\ll N_1^2,\quad L_1\ll \underline{L}_{02}.}

For the case (0), it holds that $N_1\sim N_2$.
We apply \eqref{est_lin_duh} with $b=\frac{1}{6}+$, then Corollary~\ref{cor_te_low}, to have
\eqq{\sum _{L_0,L_1,L_2}\Sc{N}^S&\lec \sum _{L_0,L_1,L_2}N_1^{s}\de ^{\frac{1}{3}-}L_1^{-\frac{1}{6}-}(L_0L_1L_2)^{\frac{1}{6}}\norm{P_{N_2,L_2}(\psi _\de u)}{L^2}\norm{P_{N_0,L_0}(\psi_\de w)}{L^2}\\
&\lec \de ^{\frac{1}{3}-}N_0^{-l}\norm{P_{N_2}(\psi _\de u)}{X^{s,\frac{1}{6},1}_S}\norm{P_{N_0}(\psi _\de w)}{X^{l,\frac{1}{6},1}_{W_{\pm}}}.}
Summing up over $N_j$ and using \eqref{est_stability} twice, we obtain the estimate for (0) with $\de ^{1-}$.

The cases (1a) -- (2b) stand for the high-low interactions, so we use Corollary~\ref{cor_te_highlow}.
In these cases we have $N_0\sim \overline{N}_{12}$ and $L_{\max} \sim \max \shugo{L_{\med},N_0^2}$.
Therefore, it holds that $L_1\lec L_{\med}$ for (1a) and (2a), and that $L_1=L_{\max}\sim N_0^2$ for (1b) and (2b).

For (1a), assume $L_2\ge L_0$ (the other case is parallel).
By \eqref{est_lin_duh} with $b=\frac{1}{4}+$, 
\eqq{\sum _{L_0,L_1,L_2}\Sc{N}^S&\lec \sum _{L_0,L_1,L_2}N_1^{s}\de ^{\frac{1}{4}-}L_1^{-\frac{1}{4}-}L_2^{\frac{1}{2}}(L_0L_1)^{\frac{1}{4}+}N_2^{\frac{d}{2}-}N_1^{-1}\norm{P_{N_2,L_2}(\psi _\de u)}{L^2}\norm{P_{N_0,L_0}(\psi _\de w)}{L^2}\\
&\lec \de ^{\frac{1}{4}-}N_1^{s-l-1}N_2^{\frac{d}{2}-s-}\norm{P_{N_2}(\psi _\de u)}{X^{s,\frac{1}{2},1}_S}\norm{P_{N_0}(\psi _\de w)}{X^{l,\frac{1}{4}+,1}_{W_{\pm}}}.}
If $(s,l)$ is in the range \eqref{hani_3d} or \eqref{hani_2d}, we have $s-l-1\le 0$ and $(s-l-1)+(\frac{d}{2}-s-)<0$.
Therefore, the above is bounded by
\eqq{\de ^{\frac{1}{4}-}N_2^{0-}\norm{P_{N_2}(\psi _\de u)}{X^{s,\frac{1}{2},1}_S}\norm{P_{N_0}(\psi _\de w)}{X^{l,\frac{1}{4}+,1}_{W_{\pm}}},}
which is summable over $N_2$ when we apply the Cauchy-Schwarz in $N_2$ to create $\tnorm{\psi _\de u}{}$.
At the end we use \eqref{est_stability} with $b=\frac{1}{4}+$ and obtain the claim.

For (1b), we first apply \eqref{est_lin_duh} with $b=\frac{1}{2}$.
The summation over $L_1$ will have no negative power of $L_1$; nevertheless, we can treat it similarly to (1a) because of the fact $L_1\sim N_0^2$.
We apply \eqref{est_stability} with $b=\frac{1}{4}+$ twice to conclude the desired estimate.

The cases (2a) and (2b) are also similar to (1a) and (1b), respectively.

Next, we treat (3a) and (3b), namely, the high-modulation high-high interactions, which is estimated with Proposition~\ref{prop_te_highlow}.
We have $N_0\lec N_1$, and again $L_1\lec L_{\med}$ for (3a), $L_1=L_{\max}\sim N_1^2$ for (3b).
The estimate for (3a) in the case $L_2\ge L_0$ is as follows:
\eqq{\sum _{L_0,L_1,L_2}\Sc{N}^S&\lec \sum _{L_0,L_1,L_2}N_1^{s}\de ^{\frac{1}{4}-}L_1^{-\frac{1}{4}-}L_2^{\frac{1}{2}}(L_0L_1)^{\frac{1}{4}}N_0^{\frac{d}{2}}N_1^{-1}\norm{P_{N_2,L_2}(\psi _\de u)}{L^2}\norm{P_{N_0,L_0}(\psi _\de w)}{L^2}\\
&\lec \de ^{\frac{1}{4}-}N_0^{\frac{d}{2}-l}N_1^{-1}\norm{P_{N_2}(\psi _\de u)}{X^{s,\frac{1}{2},1}_S}\norm{P_{N_0}(\psi _\de w)}{X^{l,\frac{1}{4},1}_{W_{\pm}}}.}
Under the assumption $\frac{d}{2}-l\le 1$, we have a prefactor $\frac{N_0}{N_1}$ which enable us to apply the Cauchy-Schwarz in $N_0$.
We conclude the estimate by applying \eqref{est_stability}.
(3b) is dealt with in the same manner, so we omit the details.

(4a) and (4b) correspond to the middle- and low-modulation interactions, and we need to consider $d\ge 3$ and $d=2$ separately.
We see that $L_1\gec L_{\med}$ for (4a), $L_1=L_{\min}$ for (4b).
As an example, we only consider the case (4b).

When $d\ge 3$, we use Proposition~\ref{prop_te_highhigh-m} and \ref{prop_te_highhigh-l}, after \eqref{est_lin_duh} with $b=\frac{1}{4}+$, to obtain
\eqq{\sum _{L_0,L_1,L_2}\Sc{N}^S\lec \de ^{\frac{1}{4}-}N_0^{\frac{d-2}{2}-l}\norm{P_{N_2}(\psi _\de u)}{X^{s,\frac{3}{8}+,1}_S}\norm{P_{N_0}(\psi _\de w)}{X^{l,\frac{3}{8}+,1}_{W_{\pm}}}.}
In order to apply the Cauchy-Schwarz in $N_0$, we have to assume $l> \frac{d}{2}-1$ with no equality.
The rest of estimate is similar to the preceding cases.

For $d=2$, we have established a trilinear estimate for the low-modulation interactions, Proposition~\ref{prop_te_highhigh-l_2d}, without division in $N_0$.
Note that the same is true for the middle-modulation interactions, since the estimate in Proposition~\ref{prop_te_highhigh-m} has the prefactor $(\frac{N_0}{N_1})^{0+}$.
Thus, instead of $\sum \Sc{N}^S$, we consider the estimate of
\eqq{\Big[ \sum _{N_1\gg 1}\Big( \sum _{N_2\sim N_1}\sum _{L_1,L_2,L_3\ll N_1^2}\norm{\psi _\de \Sc{I}_S(P_{N_1,L_1}[P_{N_2,L_2}(\psi _\de u)\cdot P_{1\ll \cdot \lec N_1,L_0}(\psi_\de w)])}{X^{s,\frac{1}{2},1}_S}\Big) ^2\Big] ^{\frac{1}{2}},}
for $d=2$.
Following the argument for $d\ge 3$, we obtain a bound
\eqq{\de ^{\frac{1}{4}-}\norm{\psi _\de u}{X^{s,\frac{3}{8}+,1}_S}\sum _{L_0\ge 1}L_0^{\frac{3}{8}+}\norm{\eta _{L_0}(\tau \pm |k|)\ti{\psi _\de w}}{L^2_\tau \ell ^2_k},}
which is sufficient whenever $l\ge 0$ if we are willing to pay a little $L_0$.

Let us next treat \eqref{est_be_w} with the same idea.
We may restrict ourselves to the case of the $+$ sign by symmetry.
We begin with the Littlewood-Paley decomposition
\eqs{\norm{\Sc{I}_{W_+}(\tfrac{\Delta _\be}{\LR{\nabla}}(u\bar{v}))}{X^{l,\frac{1}{2},1}_{W_+}(\de )}\lec \Big[ \sum _{N_0\ge 1}\Big( \sum _{N_1,N_2\ge 1}\sum _{L_0,L_1,L_2\ge 1}\Sc{N}^W\Big) ^2\Big] ^{\frac{1}{2}},\\
\Sc{N}^W=\Sc{N}^W_{N_0,N_1,N_2,L_0,L_1,L_2}:=\norm{\psi _\de \Sc{I}_{W_+}(P_{N_0,L_0}[P_{N_1,L_1}(\psi _\de u)\cdot \bbar{P_{N_2,L_2}(\psi _\de v)}])}{X^{l+1,\frac{1}{2},1}_{W_+}}.}
We will omit the detailed argument and only see how the restriction for $(s,l)$ is deduced.

The case $N_0\lec 1$ is easily estimated whenever $s\ge 0$.
For the high-low interactions, we consider, for instance, $N_0\sim N_1\gg N_2$ and $L_0=L_{\max}\sim N_0^2\gg \overline{L}_{12}$.
Imitating the above argument (case (1b) for \eqref{est_be_s}), we see that
\eqq{\sum _{L_0,L_1,L_2}\Sc{N}^W\lec N_0^{l+1}N_2^{\frac{d}{2}-}N_1^{-1}\cdot N_1^{-s}N_2^{-s}\norm{P_{N_1}(\psi _\de u)}{X^{s,\frac{1}{4}+,1}_S}\norm{P_{N_2}(\psi _\de v)}{X^{s,\frac{1}{4}+,1}_S}.}
This is appropriately estimated under the assumption $l-s\le 0$, $-2s+l+\frac{d}{2}\le 0$.
For the high-modulation high-high interactions, considering the case $1\ll N_0\lec N_1\sim N_2$ and $L_1=L_{\max}\gec N_1^2$ for example, we obtain
\eqq{\sum _{L_0,L_1,L_2}\Sc{N}^W\lec \de ^{\frac{1}{4}-}N_0^{l+1}N_0^{\frac{d}{2}}N_1^{-1}\cdot N_1^{-s}N_2^{-s}\norm{P_{N_1}(\psi _\de u)}{X^{s,\frac{1}{2},1}_S}\norm{P_{N_2}(\psi _\de v)}{X^{s,\frac{1}{4},1}_S}.}
Since $2s+1>0$ and $l+\frac{d}{2}+1\le 2s+1$ under the assumption $2s\ge l+\frac{d}{2}\ge d-1$, this is summable over $N_0$.
Finally, we consider particularly $1\ll N_0\lec N_1\sim N_2$, $L_{\max}\ll N_1^2$, and $L_2=L_{\min}$, as an example of the high-high interactions with middle or low modulation.
We obtain
\eqq{\sum _{L_0,L_1,L_2}\Sc{N}^W\lec \de ^{\frac{1}{8}-}N_0^{l+1}N_0^{\frac{d-2}{2}}\cdot N_1^{-s}N_2^{-s}\norm{P_{N_1}(\psi _\de u)}{X^{s,\frac{3}{8}+,1}_S}\norm{P_{N_2}(\psi _\de v)}{X^{s,\frac{1}{4},1}_S},}
where we still have enough negative power of $N_1$, since $2s>0$ holds under our assumption.
Therefore, in contrast to the Schr\"odinger estimate \eqref{est_be_s}, the wave bilinear estimate \eqref{est_be_w} admits the border case $2s=l+\frac{d}{2}$ even for $d\ge 3$.
In fact, for $d=2$ we do not have to care about the decomposition with respect to $N_0$.
\end{proof}

\begin{proof}[Proof of Theorem~\ref{thm_main3}]
We write \eqref{IE} as $(u,w)(t)=\Phi _{(u_0,w_0)}(u,w)(t)$.
For the term $\LR{\nabla}^{-1}(w+\bar{w})$, we use \eqref{est_lin_duh} and \eqref{est_stability} to verify
\eqq{&\norm{\int _0^te^{-i(t-t')\LR{\nabla}}\LR{\nabla}^{-1}w(t')\,dt'}{X^{l,\frac{1}{2},1}_{W_+}(\de )}\le \norm{\psi _\de \int _0^te^{-i(t-t')\LR{\nabla}}\LR{\nabla}^{-1}(\psi _\de w)(t')\,dt'}{X^{l,\frac{1}{2},1}_{W_+}}\\
&\lec \de ^{\frac{1}{2}-}\norm{w}{X^{l-1,0,1}_{W_+}}\lec \de ^{1-}\norm{w}{X^{l,\frac{1}{2},1}_{W_+}}.}
Taking infimun over $w$, we have
\eqq{\norm{\int _0^te^{-i(t-t')\LR{\nabla}}\LR{\nabla}^{-1}w(t')\,dt'}{X^{l,\frac{1}{2},1}_{W_+}(\de )}\lec \de ^{1-}\norm{w}{X^{l,0,1}_{W_+}(\de )}.}
We also have a similar estimate for $\bar{w}$, since
\eqq{\norm{\bar{w}}{X^{l,0-,1}_{W_+}}\lec \norm{\bar{w}}{X^{l,0,2}_{W_+}}=\norm{w}{X^{l,0,2}_{W_+}}\le \norm{w}{X^{l,0,1}_{W_+}}.}

Let $(s,l)$ be such that \eqref{hani_3d} or \eqref{hani_2d} is true, and $r>0$ be any radius.
For any $(u_0,w_0)\in H^s\times H^l$ satisfying $\tnorm{(u_0,w_0)}{H^s\times H^l}\le r$, we see from \eqref{est_lin_sol}, \eqref{est_be_s}, and \eqref{est_be_w} that
\eqq{&\norm{\Phi _{(u_0,w_0)}(u,w)}{X^{s,\frac{1}{2},1}_S(\de )\times X^{l,\frac{1}{2},1}_{W_+}(\de )}\\
&\le Cr+C\de ^{\frac{1}{2}-}\big( \norm{(u,w)}{X^{s,\frac{1}{2},1}_S(\de )\times X^{l,\frac{1}{2},1}_{W_+}(\de )}+\norm{(u,w)}{X^{s,\frac{1}{2},1}_S(\de )\times X^{l,\frac{1}{2},1}_{W_+}(\de )}^2\big)}
for $0<\de \le 1$, which implies that $\Phi _{(u_0,w_0)}$ is a map on the ball of radius $2Cr$ in $X^{s,\frac{1}{2},1}_S(\de )\times X^{l,\frac{1}{2},1}_{W_+}(\de )$ centered at the origin, provided $\de ^{\frac{1}{2}-}r\ll 1$.
Similarly, we have
\eqq{&\norm{\Phi _{(u_0,w_0)}(u,w)-\Phi _{(u_0,w_0)}(u',w')}{X^{s,\frac{1}{2},1}_S(\de )\times X^{l,\frac{1}{2},1}_{W_+}(\de )}\\
&\le C\de ^{\frac{1}{2}-}\big( 1+\norm{(u,w)}{X^{s,\frac{1}{2},1}_S(\de )\times X^{l,\frac{1}{2},1}_{W_+}(\de )}+\norm{(u',w')}{X^{s,\frac{1}{2},1}_S(\de )\times X^{l,\frac{1}{2},1}_{W_+}(\de )}\big) \\
&\hspace{40pt}\times \norm{(u,w)-(u',w')}{X^{s,\frac{1}{2},1}_S(\de )\times X^{l,\frac{1}{2},1}_{W_+}(\de )},}
which shows that $\Phi _{(u_0,w_0)}$ is contractive on this ball provided $\de ^{\frac{1}{2}-}r\ll 1$, giving a solution to \eqref{IE}.
The uniqueness of solution in the whole function space $X^{s,\frac{1}{2},1}_S(\de )\times X^{l,\frac{1}{2},1}_{W_+}(\de )$ and the Lipschitz continuity of the data-to-solution map then follow from a standard argument.
\end{proof}


\section{Proof of Theorem~\ref{thm_illp}}\label{sec_pf-illp}

In this section we shall verify Theorem~\ref{thm_illp}.
More precisely, we will show the following.

\begin{thm}\label{thm_illp2}
Let $\la ,c_0,\al ,\be ,\ga$ be any constants as \eqref{constants}.
The following holds.

(i) [Norm inflation] Assume that $d=2$, $s<\frac{3}{2}$, $l>\max \shugo{0,\,2s-1}$.
Then, there exists a sequence $\shugo{u_{0,N}}$ of smooth functions on $\Tg^2$ satisfying $\tnorm{u_{0,N}}{H^s}\to 0$ as $N\to \I$, such that the solution $(u_N,n_N)$ to \eqref{Z} with initial data $(u_{0,N},0,0)$ satisfies $\tnorm{n_N(t)}{H^l}\to \I$ as $N\to \I$ for any $t\in \R$, $0<|t|\ll 1$.

(ii) [Non-existence of continuous map] Assume that $d=2$, $s<\frac{1}{2}$, $l=0$.
Then, there exists a sequence $\shugo{u_{0,N}}$ of smooth functions on $\Tg^2$ satisfying $\tnorm{u_{0,N}}{H^s}\to 0$ as $N\to \I$, such that the solution $(u_N,n_N)$ to \eqref{Z} with initial data $(u_{0,N},0,0)$ satisfies $\tnorm{n_N(t)}{H^l}\sim 1$ for any $t\in \R$, $0<|t|\ll 1$ and any sufficiently large $N$.

(iii) [Non-existence of $C^2$ map] Let $d\ge 2$.
Assume that either $l<\max \shugo{0,\,s-2}$ or $l>\min \shugo{2s-1,\,s+1}$ holds, and that the data-to-solution map $(u_0,n_0,n_1)\mapsto (u,n)$ of \eqref{Z} for smooth data extends to a continuous map
\eqq{&\Shugo{(u_0,n_0,n_1)\in H^{s,l}}{\tnorm{(u_0,n_0,n_1)}{H^{s,l}}\le R}\quad \to \quad \mathcal{C}([-T,T];H^{s,l})}
for some $R,T>0$.
Then, this map will not be $C^2$ in these topologies at the origin.

(iv) [Lack of bilinear estimates in the Bourgain spaces] Let $d\ge 2$.
Then, the bilinear estimates 
\begin{gather}
\norm{uw}{X^{s,b-1,p}_S}\lec \norm{u}{X^{s,b,p}_S}\norm{w}{X^{l,b,p}_{W_{\pm}}},\label{behanrei1}\\
\norm{\frac{\Delta _\be}{\LR{\nabla}}(u\bar{v})}{X^{l,b-1,p}_{W_{\pm}}}\lec \norm{u}{X^{s,b,p}_S}\norm{v}{X^{s,b,p}_S}\label{behanrei2}
\end{gather}
do not hold for any $b\in \R$, $1\le p<\I$ if $l<\max \shugo{0,\,s-1}$ and if $l>\min \shugo{2s-1,\,s}$, respectively.
\end{thm}

\begin{rem}
(i)--(ii) means the ill-posedness of the problem, since the data-to-solution map on smooth data cannot extend to a continuous map under these regularities.
The norm-inflation phenomena like (i) was observed for the Zakharov system on $\R$ by Holmer~\cite{Ho}, and we will take the same approach.
The ill-posedness assertion like (ii) was mentioned in Bejenaru and Tao's work~\cite{BT} in a general framework; see also \cite{KiT} for related results.

From (iii), we can say that the usual contraction argument in any space embedded continuously into $\mathcal{C}([-T,T];H^{s,l})$ does not work in these regularities.
Results of this type, which does not directly mean the ill-posedness of the problem, was first given by Bourgain~\cite{B97} in the context of the Korteweg-de Vries equation.

The bilinear estimates stated in (iv), which (with a suitable $b$ and $p$) yield the local well-posedness for small initial data, are easily deduced from trilinear estimates obtained in Section~\ref{sec_trilinear} provided $(s,l)$ is in the range \eqref{hani_3d} or \eqref{hani_2d}.
Note that in the 2d case the regularity range given in (iv) exactly complements the range \eqref{hani_2d}.
The claim (iv) still holds for the Bourgain spaces of $\ell ^\I$-Besov type $X^{s,b,\I}$ (defined in a natural way), which is trivial from the proof below.
The lack of these estimates still prevent us from the usual contraction argument in the Bourgain spaces.
However, in some regularity range it is strongly expected that a suitable modification of the Bourgain spaces will restore the bilinear estimates which will yield the local well-posedness of the problem.
\end{rem}

We start the proof of Theorem~\ref{thm_illp2} by (iii).
The initial value problem \eqref{Z} is replaced by the system of integral equations
\eq{IE:Z}{
u(t)&=e^{it\Delta}u_0-i\la \int _0^te^{i(t-t')\Delta}\big[ n(t')u(t')\big] \,dt',\\
n(t)&=\cos (t|\nabla |)n_0+\frac{\sin (t|\nabla |)}{|\nabla |}n_1-\int _0^t\frac{\sin ((t-t')|\nabla |)}{|\nabla |}\Delta _\be \big[ u(t')\bbar{u(t')}\big] \,dt',
}
after the normalization of constants such that $c_0=1$, $\al =(1,\dots ,1)$.
We focus on the quadratic terms in the iteration scheme,
\eqq{u^{(2)}[u_0,n_0,n_1](t)&=-i\la \int _0^te^{i(t-t')\Delta}\left[ \left( \cos (t'|\nabla |)n_0+\frac{\sin (t'|\nabla |)}{|\nabla |}n_1\right) e^{it'\Delta}u_0\right] \,dt',\\
n^{(2)}[u_0](t)&=-\int _0^t\frac{\sin ((t-t')|\nabla |)}{|\nabla |}\Delta _\be \big[ e^{it'\Delta}u_0\bbar{e^{it'\Delta}u_0}\big] \,dt'.}

Throughout this section we assume $\be _1\neq 0$ for the constant $\be =(\be _1,\dots ,\be _d)\in \R ^d\setminus \shugo{(0,\dots ,0)}$.
For $d\ge 2$ and $1\ll N\in \Bo{N}$, define $K_N,\ti{K}_N\in \Zg ^d$ as
\eqq{K_N=(\frac{N}{\ga _1},\frac{n-1}{\ga _2},0,\dots ,0),\quad  \ti{K}_N=(\frac{1-N}{\ga _1},\frac{n}{\ga _2},0,\dots ,0),}
where $n$ is the unique integer satisfying
\eq{assump:res}{|K_N-\ti{K}_N|+|K_N|^2-|\ti{K}_N|^2=\sqrt{\frac{(2N-1)^2}{\ga _1^2}+\frac{1}{\ga _2^2}}+\frac{2N-1}{\ga _1^2}-\frac{2n-1}{\ga _2^2} \in \left( -\frac{1}{\ga _2^2},\frac{1}{\ga _2^2}\right] .}
Note that $|K_N|\sim |\ti{K}_N|\sim |K_N-\ti{K}_N|\sim N$ and 
\eq{assump:nonres}{\Big| -|K_N-\ti{K}_N|+|K_N|^2-|\ti{K}_N|^2\Big| \sim N.}
We set
\eqs{f _N(x):=e^{iK_N\cdot x}+e^{i\ti{K}_N\cdot x},\qquad g _N(x):=\cos ((K_N-\ti{K}_N)\cdot x).}

Now, the claim (iii) will be verified from the following lemma.
We refer to \cite{Ho} for the detailed argument.
\begin{lem}\label{lem:notc2}
We have the following.

(i) Let $(s,l)\in \R ^2$ satisfy $l<0$.
Then, there exists $t_0>0$ such that the estimate
\eq{est:bdd-quad-s}{\norm{u^{(2)}[u_0,n_0,n_1](t)}{H^s}\lec \norm{(u_0,n_0,n_1)}{H^{s,l}}^2,\quad \forall (u_0,n_0,n_1)\in H^{s,l}}
fails for any $t\in (-t_0,0)\cup (0,t_0)$.

(ii) Let $(s,l)\in \R^2$ satisfy $l>2s-1$.
Then, there exists $t_0>0$ such that the estimate
\eq{est:bdd-quad-w}{\norm{n^{(2)}[u_0](t)}{H^l}\lec \norm{u_0}{H^s}^2,\quad \forall u_0\in H^s}
fails for any $t\in (-t_0,0)\cup (0,t_0)$.

(iii) Let $(s,l)\in \R ^2$ satisfy $l<s-2$.
Then, the estimate
\eq{est:bdd-quad-s2}{\sup _{-T\le t\le T}\norm{u^{(2)}[u_0,n_0,n_1](t)}{H^s}\lec \norm{(u_0,n_0,n_1)}{H^{s,l}}^2,\quad \forall (u_0,n_0,n_1)\in H^{s,l}}
fails for any $T>0$.

(iv) Let $(s,l)\in \R^2$ satisfy $l>s+1$.
Then, the estimate
\eq{est:bdd-quad-w2}{\sup _{-T\le t\le T}\norm{n^{(2)}[u_0](t)}{H^l}\lec \norm{u_0}{H^s}^2,\quad \forall u_0\in H^s}
fails for any $T>0$.
\end{lem}

\begin{proof}
(i) Set $u_0:=N^{-s}f _N$, $n_0:=N^{-l}g _N$, $n_1=0$ for large $N\in \Bo{N}$.
Then, it holds that $\tnorm{(u_0,n_0,n_1)}{H^{s,l}}\sim 1$.
On the other hand, a direct calculation shows that
\eqq{\F _xu^{(2)}[u_0,n_0,0](t,K_N)=cN^{-s-l}\int _0^te^{-i(t-t')|K_N|^2}\left( e^{it'|K_N-\ti{K}_N|}+e^{-it'|K_N-\ti{K}_N|}\right) e^{-it'|\ti{K}_N|^2}\,dt',}
and, by \eqref{assump:res} and \eqref{assump:nonres}, that
\eqq{&|\F _xu^{(2)}[u_0,n_0,0](t,K_N)|\\
&\ge cN^{-s-l}\left( \left| \int _0^te^{it'\big( |K_N-\ti{K}_N|+|K_N|^2-|\ti{K}_N|^2\big)}\,dt'\right| -\left| \int _0^te^{it'\big( -|K_N-\ti{K}_N|+|K_N|^2-|\ti{K}_N|^2\big)}\,dt'\right| \right) \\
&\ge c|t|N^{-s-l}-c'N^{-s-l-1}\ge c|t|N^{-s-l}}
for $0<|t|\le t_0\sim 1$ (for instance $t_0=\frac{\ga _2^2}{100}$) and $N\gg |t|^{-1}$.
Therefore, we obtain
\eqq{\norm{u^{(2)}[u_0,n_0,0](t)}{H^s}&\gec N^s|\F _xu^{(2)}[u_0,n_0,0](t,K_N)|\gec |t|N^{-l},}
which implies that the estimate \eqref{est:bdd-quad-s} does not hold for all $N$ provided $l<0$.

(ii) We use $u_0:=N^{-s}f _N$ and make a similar argument.
It follows that
\eqq{&|\F _xn^{(2)}[u_0](t,K_N-\ti{K}_N)|\\
&=cN^{-2s}\left| \int _0^t \frac{e^{i(t-t')|K_N-\ti{K}_N|}-e^{-i(t-t')|K_N-\ti{K}_N|}}{|K_N-\ti{K}_N|}|K_N-\ti{K}_N|_\be ^2e^{-it'|K_N|^2}e^{it'|\ti{K}_N|^2}\,dt'\right| ,}
where $|K_N-\ti{K}_N|_\be ^2:=\be _1(\frac{2N-1}{\ga _1})^2+\be _2(\frac{-1}{\ga _2})^2$.
Noting $\be _1\neq 0$, we obtain the lower bound of the above as
\eqq{&cN^{-2s+1}\left( \left| \int _0^te^{it'\big( -|K_N-\ti{K}_N|-|K_N|^2+|\ti{K}_N|^2\big)}\,dt'\right| -\left| \int _0^te^{it'\big( |K_N-\ti{K}_N|-|K_N|^2+|\ti{K}_N|^2\big)}\,dt'\right| \right) \\
&\ge c|t|N^{-2s+1}-c'N^{-2s}\ge c|t|N^{-2s+1},}
for $0<|t|\le t_0\sim 1$ and $N\gg |t|^{-1}$.
Hence we have
\eqq{\norm{n^{(2)}[u_0](t)}{H^l}&\gec |t|N^{l-2s+1},}
and the estimate \eqref{est:bdd-quad-w} does not hold for all $N$ provided $l-2s+1>0$.

(iii) Consider the following initial data with $H^{s,l}$ norm $\sim 1$; $u_0:=1$, $n_0(x):=N^{-l}\cos (\frac{N}{\ga _1}x_1)$, $n_1:=0$ for large $N\in \Bo{N}$.
We see that
\eqq{\F _xu^{(2)}[u_0,n_0,n_1](t,(\frac{N}{\ga _1},0,\dots ,0))=cN^{-l}e^{-it(\frac{N}{\ga _1})^2}\int _0^te^{it'(\frac{N}{\ga _1})^2}\cos (t'\frac{N}{\ga _1})\,dt'.}
Taking $t=\frac{\ga _1^2}{100N^2}$, we have $\Re \left[ e^{it'(\frac{N}{\ga _1})^2}\cos (t'\frac{N}{\ga _1})\right] \ge \frac{1}{2}$ for $0<t'<t$, obtaining
\eqq{\big| \F _xu^{(2)}[u_0,n_0,n_1](\frac{\ga _1^2}{100N^2},(\frac{N}{\ga _1},0,\dots ,0))\big| \ge cN^{-l-2}.}
Hence, it holds that
\eqq{\sup _{-T\le t\le T}\norm{u^{(2)}[u_0,n_0,n_1](t)}{H^s}\gec N^{s-l-2}}
for $T>0$ and $N>(\frac{\ga _1^2}{100T})^{1/2}$.
Therefore, \eqref{est:bdd-quad-s2} does not hold if $s-l-2>0$.

(iv) Set $u_0:=1+N^{-s}e^{i\frac{N}{\ga _1}x_1}$, which has an $H^s$ norm $\sim 1$.
Some calculation shows
\eqq{&\F _xn^{(2)}[u_0](t,(\frac{N}{\ga _1},0,\dots ,0))=c\int _0^t \frac{\sin \left( (t-t')\frac{N}{\ga _1}\right)}{N/\ga _1}\be _1(\frac{N}{\ga _1})^2e^{-it'(\frac{N}{\ga _1})^2}N^{-s}\,dt'\\
&=cN^{1-s}\left\{ \frac{e^{-it(\frac{N}{\ga _1})^2}-e^{it\frac{N}{\ga _1}}}{(\frac{N}{\ga _1})^2+\frac{N}{\ga _1}}-\frac{e^{-it(\frac{N}{\ga _1})^2}-e^{-it\frac{N}{\ga _1}}}{(\frac{N}{\ga _1})^2-\frac{N}{\ga _1}}\right\}\\
&=cN^{1-s}\frac{1}{(\frac{N}{\ga _1})^2\{ (\frac{N}{\ga _1})^2-1\}}\left\{ -2N^2i\sin (t\frac{N}{\ga _1}) +2N\cos (t\frac{N}{\ga _1}) -2Ne^{-it(\frac{N}{\ga _1})^2}\right\} .}
Then, taking $t=\frac{\pi \ga _1}{2N}$, we see that the first term in the last line above dominates the rest and
\eqq{\big| \F _xn^{(2)}[u_0](\frac{\pi \ga _1}{2N},(\frac{N}{\ga _1},0,\dots ,0))\big| \ge cN^{-s-1},}
hence
\eqq{\sup _{-T\le t\le T}\norm{n^{(2)}[u_0](t)}{H^l}\gec N^{l-s-1}}
for $T>0$ and $N>\frac{\pi \ga _1}{2T}$, concluding that \eqref{est:bdd-quad-w2} does not hold if $l-s-1>0$.
\end{proof}

Next, we show the wave norm-inflation phenomena (i) employing the argument of Holmer~\cite{Ho}.
We also show (ii) as a by-product of the proof.

We consider the case $s<1$ first.
Take an arbitrary $(s',l')$ such that $s<s'$, $l\ge l'$ and $0<2s'-1<l'<1$, then set $s_+$ so that $l'=2s_+-1$ (see Figure~3, the left one).
\begin{figure}\label{fig:choice}
\begin{center}
\hspace*{-50pt}
\unitlength 0.1in
\begin{picture}( 33.6000, 32.0000)( -5.6000,-34.0000)
\put(1.4500,-2.0000){\makebox(0,0)[rt]{$l$}}%
%
{\color[named]{Black}{%
\special{pn 8}%
\special{pa 200 3400}%
\special{pa 200 200}%
\special{fp}%
\special{sh 1}%
\special{pa 200 200}%
\special{pa 180 268}%
\special{pa 200 254}%
\special{pa 220 268}%
\special{pa 200 200}%
\special{fp}%
\special{pa 0 3200}%
\special{pa 2800 3200}%
\special{fp}%
\special{sh 1}%
\special{pa 2800 3200}%
\special{pa 2734 3180}%
\special{pa 2748 3200}%
\special{pa 2734 3220}%
\special{pa 2800 3200}%
\special{fp}%
}}%
%
{\color[named]{Black}{%
\special{pn 13}%
\special{pa 2800 1800}%
\special{pa 1400 3200}%
\special{fp}%
\special{pa 1400 3200}%
\special{pa 800 3200}%
\special{fp}%
\special{pa 800 3200}%
\special{pa 1400 2000}%
\special{fp}%
\special{pa 1400 2000}%
\special{pa 2800 600}%
\special{fp}%
}}%
%
{\color[named]{Black}{%
\special{pn 8}%
\special{pa 800 3200}%
\special{pa 800 2000}%
\special{dt 0.045}%
\special{pa 800 2000}%
\special{pa 1400 2000}%
\special{dt 0.045}%
\special{pa 1400 2000}%
\special{pa 1400 200}%
\special{dt 0.045}%
}}%
\put(2.3000,-32.4000){\makebox(0,0)[lt]{$0$}}%
\put(6.9000,-32.4000){\makebox(0,0)[lt]{$1/2$}}%
\put(13.7000,-32.4000){\makebox(0,0)[lt]{$1$}}%
\put(1.6000,-31.6000){\makebox(0,0)[rb]{$0$}}%
\put(1.6000,-19.4000){\makebox(0,0)[rt]{$1$}}%
%
{\color[named]{Black}{%
\special{pn 13}%
\special{pa 180 2000}%
\special{pa 220 2000}%
\special{fp}%
\special{pa 800 3220}%
\special{pa 800 3180}%
\special{fp}%
\special{pa 1400 3180}%
\special{pa 1400 3220}%
\special{fp}%
}}%
\put(28.0000,-32.4000){\makebox(0,0)[rt]{$s$}}%
%
{\color[named]{Black}{%
\special{pn 4}%
\special{sh 1}%
\special{ar 600 800 14 14 0  6.28318530717959E+0000}%
\special{sh 1}%
\special{ar 900 2300 14 14 0  6.28318530717959E+0000}%
\special{sh 1}%
\special{ar 1250 2300 14 14 0  6.28318530717959E+0000}%
\special{sh 1}%
\special{ar 1250 3200 14 14 0  6.28318530717959E+0000}%
\special{sh 1}%
\special{ar 800 3200 14 14 0  6.28318530717959E+0000}%
}}%
\put(6.4500,-7.5000){\makebox(0,0)[rb]{{\footnotesize $(s,l)$}}}%
\put(9.5000,-22.6000){\makebox(0,0)[rb]{{\colorbox[named]{White}{\color[named]{Black}{{\footnotesize $(s',l')$}}}}}}%
\put(8.0000,-31.7500){\makebox(0,0)[rb]{{\footnotesize $(1/2,0)$}}}%
%
{\color[named]{Black}{%
\special{pn 8}%
\special{pa 2800 1800}%
\special{pa 2800 600}%
\special{ip}%
}}%
%
{\color[named]{Black}{%
\special{pn 4}%
\special{pa 2800 1130}%
\special{pa 876 3056}%
\special{fp}%
\special{pa 2800 1160}%
\special{pa 846 3116}%
\special{fp}%
\special{pa 2800 1190}%
\special{pa 816 3176}%
\special{fp}%
\special{pa 2800 1220}%
\special{pa 826 3196}%
\special{fp}%
\special{pa 2800 1250}%
\special{pa 856 3196}%
\special{fp}%
\special{pa 2800 1280}%
\special{pa 886 3196}%
\special{fp}%
\special{pa 2800 1310}%
\special{pa 916 3196}%
\special{fp}%
\special{pa 2800 1340}%
\special{pa 946 3196}%
\special{fp}%
\special{pa 2800 1370}%
\special{pa 976 3196}%
\special{fp}%
\special{pa 2800 1400}%
\special{pa 1006 3196}%
\special{fp}%
\special{pa 2800 1430}%
\special{pa 1036 3196}%
\special{fp}%
\special{pa 2800 1460}%
\special{pa 1066 3196}%
\special{fp}%
\special{pa 2800 1490}%
\special{pa 1096 3196}%
\special{fp}%
\special{pa 2800 1520}%
\special{pa 1126 3196}%
\special{fp}%
\special{pa 2800 1550}%
\special{pa 1156 3196}%
\special{fp}%
\special{pa 2800 1580}%
\special{pa 1186 3196}%
\special{fp}%
\special{pa 2800 1610}%
\special{pa 1216 3196}%
\special{fp}%
\special{pa 2800 1640}%
\special{pa 1250 3190}%
\special{fp}%
\special{pa 2800 1670}%
\special{pa 1276 3196}%
\special{fp}%
\special{pa 2800 1700}%
\special{pa 1306 3196}%
\special{fp}%
\special{pa 2800 1730}%
\special{pa 1336 3196}%
\special{fp}%
\special{pa 2800 1760}%
\special{pa 1366 3196}%
\special{fp}%
\special{pa 2800 1100}%
\special{pa 906 2996}%
\special{fp}%
\special{pa 2800 1070}%
\special{pa 936 2936}%
\special{fp}%
\special{pa 2800 1040}%
\special{pa 966 2876}%
\special{fp}%
\special{pa 2800 1010}%
\special{pa 996 2816}%
\special{fp}%
\special{pa 2800 980}%
\special{pa 1026 2756}%
\special{fp}%
\special{pa 2800 950}%
\special{pa 1056 2696}%
\special{fp}%
\special{pa 2800 920}%
\special{pa 1086 2636}%
\special{fp}%
\special{pa 2800 890}%
\special{pa 1116 2576}%
\special{fp}%
}}%
%
{\color[named]{Black}{%
\special{pn 4}%
\special{pa 2800 860}%
\special{pa 1146 2516}%
\special{fp}%
\special{pa 2800 830}%
\special{pa 1176 2456}%
\special{fp}%
\special{pa 2800 800}%
\special{pa 1206 2396}%
\special{fp}%
\special{pa 2800 770}%
\special{pa 1236 2336}%
\special{fp}%
\special{pa 2800 740}%
\special{pa 1266 2276}%
\special{fp}%
\special{pa 2800 710}%
\special{pa 1296 2216}%
\special{fp}%
\special{pa 2800 680}%
\special{pa 1326 2156}%
\special{fp}%
\special{pa 2800 650}%
\special{pa 1356 2096}%
\special{fp}%
\special{pa 2800 620}%
\special{pa 1386 2036}%
\special{fp}%
}}%
\put(13.0000,-31.5000){\makebox(0,0)[lb]{{\colorbox[named]{White}{\color[named]{Black}{{\footnotesize $(s_+,0)$}}}}}}%
\put(15.5000,-22.6000){\makebox(0,0)[rb]{{\colorbox[named]{White}{\color[named]{Black}{{\footnotesize $(s_+,l')$}}}}}}%
%
{\color[named]{Black}{%
\special{pn 13}%
\special{pa 926 2300}%
\special{pa 1226 2300}%
\special{fp}%
\special{sh 1}%
\special{pa 1226 2300}%
\special{pa 1158 2280}%
\special{pa 1172 2300}%
\special{pa 1158 2320}%
\special{pa 1226 2300}%
\special{fp}%
\special{pa 1250 2326}%
\special{pa 1250 3176}%
\special{fp}%
\special{sh 1}%
\special{pa 1250 3176}%
\special{pa 1270 3108}%
\special{pa 1250 3122}%
\special{pa 1230 3108}%
\special{pa 1250 3176}%
\special{fp}%
\special{pa 1226 3200}%
\special{pa 826 3200}%
\special{fp}%
\special{sh 1}%
\special{pa 826 3200}%
\special{pa 892 3220}%
\special{pa 878 3200}%
\special{pa 892 3180}%
\special{pa 826 3200}%
\special{fp}%
}}%
\end{picture}%
 \hspace{0pt}
\unitlength 0.1in
\begin{picture}( 29.6000, 32.0000)( -1.6000,-34.0000)
%
{\color[named]{Black}{%
\special{pn 4}%
\special{pa 2790 1320}%
\special{pa 916 3196}%
\special{fp}%
\special{pa 2790 1290}%
\special{pa 886 3196}%
\special{fp}%
\special{pa 2790 1260}%
\special{pa 856 3196}%
\special{fp}%
\special{pa 2790 1230}%
\special{pa 826 3196}%
\special{fp}%
\special{pa 2790 1200}%
\special{pa 816 3176}%
\special{fp}%
\special{pa 2790 1170}%
\special{pa 846 3116}%
\special{fp}%
\special{pa 2790 1140}%
\special{pa 876 3056}%
\special{fp}%
\special{pa 2790 1110}%
\special{pa 906 2996}%
\special{fp}%
\special{pa 2790 1080}%
\special{pa 936 2936}%
\special{fp}%
\special{pa 2790 1050}%
\special{pa 966 2876}%
\special{fp}%
\special{pa 2790 1020}%
\special{pa 996 2816}%
\special{fp}%
\special{pa 2790 990}%
\special{pa 1026 2756}%
\special{fp}%
\special{pa 2790 960}%
\special{pa 1056 2696}%
\special{fp}%
\special{pa 2790 930}%
\special{pa 1086 2636}%
\special{fp}%
\special{pa 2790 900}%
\special{pa 1116 2576}%
\special{fp}%
\special{pa 2790 870}%
\special{pa 1146 2516}%
\special{fp}%
\special{pa 2790 840}%
\special{pa 1176 2456}%
\special{fp}%
\special{pa 2790 810}%
\special{pa 1206 2396}%
\special{fp}%
\special{pa 2790 780}%
\special{pa 1236 2336}%
\special{fp}%
\special{pa 2790 750}%
\special{pa 1266 2276}%
\special{fp}%
\special{pa 2790 720}%
\special{pa 1296 2216}%
\special{fp}%
\special{pa 2790 690}%
\special{pa 1326 2156}%
\special{fp}%
\special{pa 2790 660}%
\special{pa 1356 2096}%
\special{fp}%
\special{pa 2790 630}%
\special{pa 1386 2036}%
\special{fp}%
\special{pa 2790 1350}%
\special{pa 946 3196}%
\special{fp}%
\special{pa 2790 1380}%
\special{pa 976 3196}%
\special{fp}%
\special{pa 2790 1410}%
\special{pa 1006 3196}%
\special{fp}%
\special{pa 2790 1440}%
\special{pa 1036 3196}%
\special{fp}%
\special{pa 2790 1470}%
\special{pa 1066 3196}%
\special{fp}%
\special{pa 2790 1500}%
\special{pa 1096 3196}%
\special{fp}%
}}%
%
{\color[named]{Black}{%
\special{pn 4}%
\special{pa 2790 1530}%
\special{pa 1126 3196}%
\special{fp}%
\special{pa 2790 1560}%
\special{pa 1156 3196}%
\special{fp}%
\special{pa 2790 1590}%
\special{pa 1186 3196}%
\special{fp}%
\special{pa 2790 1620}%
\special{pa 1216 3196}%
\special{fp}%
\special{pa 2790 1650}%
\special{pa 1250 3190}%
\special{fp}%
\special{pa 2790 1680}%
\special{pa 1276 3196}%
\special{fp}%
\special{pa 2790 1710}%
\special{pa 1306 3196}%
\special{fp}%
\special{pa 2790 1740}%
\special{pa 1336 3196}%
\special{fp}%
\special{pa 2790 1770}%
\special{pa 1366 3196}%
\special{fp}%
}}%
\put(1.4500,-2.0000){\makebox(0,0)[rt]{$l$}}%
%
{\color[named]{Black}{%
\special{pn 8}%
\special{pa 200 3400}%
\special{pa 200 200}%
\special{fp}%
\special{sh 1}%
\special{pa 200 200}%
\special{pa 180 268}%
\special{pa 200 254}%
\special{pa 220 268}%
\special{pa 200 200}%
\special{fp}%
\special{pa 0 3200}%
\special{pa 2800 3200}%
\special{fp}%
\special{sh 1}%
\special{pa 2800 3200}%
\special{pa 2734 3180}%
\special{pa 2748 3200}%
\special{pa 2734 3220}%
\special{pa 2800 3200}%
\special{fp}%
}}%
%
{\color[named]{Black}{%
\special{pn 13}%
\special{pa 2800 1800}%
\special{pa 1400 3200}%
\special{fp}%
\special{pa 1400 3200}%
\special{pa 800 3200}%
\special{fp}%
\special{pa 800 3200}%
\special{pa 1400 2000}%
\special{fp}%
\special{pa 1400 2000}%
\special{pa 2800 600}%
\special{fp}%
}}%
\put(2.4000,-32.4000){\makebox(0,0)[lt]{$0$}}%
\put(6.9000,-32.4000){\makebox(0,0)[lt]{$1/2$}}%
\put(13.7000,-32.4000){\makebox(0,0)[lt]{$1$}}%
\put(1.6000,-31.6000){\makebox(0,0)[rb]{$0$}}%
\put(1.6000,-19.4000){\makebox(0,0)[rt]{$1$}}%
%
{\color[named]{Black}{%
\special{pn 13}%
\special{pa 180 2000}%
\special{pa 220 2000}%
\special{fp}%
\special{pa 800 3220}%
\special{pa 800 3180}%
\special{fp}%
\special{pa 1400 3180}%
\special{pa 1400 3220}%
\special{fp}%
}}%
\put(28.0000,-32.4000){\makebox(0,0)[rt]{$s$}}%
%
{\color[named]{Black}{%
\special{pn 13}%
\special{pa 180 1400}%
\special{pa 220 1400}%
\special{fp}%
}}%
%
{\color[named]{Black}{%
\special{pn 13}%
\special{pa 180 800}%
\special{pa 220 800}%
\special{fp}%
}}%
%
{\color[named]{Black}{%
\special{pn 13}%
\special{pa 2000 3180}%
\special{pa 2000 3220}%
\special{fp}%
}}%
\put(1.6000,-7.4000){\makebox(0,0)[rt]{$2$}}%
\put(14.5000,-3.6000){\makebox(0,0)[lb]{{\footnotesize $(s,l)$}}}%
\put(18.5500,-9.6000){\makebox(0,0)[rb]{{\footnotesize $(s',l')$}}}%
\put(12.6000,-22.0000){\makebox(0,0)[rb]{{\footnotesize $(s_-,l_-)$}}}%
\put(24.4000,-22.0000){\makebox(0,0)[lt]{{\footnotesize $(s_+,l_-)$}}}%
\put(24.4000,-9.6000){\makebox(0,0)[rb]{{\footnotesize $(s_+,l')$}}}%
\put(18.9000,-32.4000){\makebox(0,0)[lt]{$3/2$}}%
%
{\color[named]{Black}{%
\special{pn 8}%
\special{pa 1400 2000}%
\special{pa 1400 200}%
\special{dt 0.045}%
\special{pa 1400 2000}%
\special{pa 2000 800}%
\special{dt 0.045}%
\special{pa 2000 800}%
\special{pa 2000 200}%
\special{dt 0.045}%
\special{pa 2000 800}%
\special{pa 1400 1400}%
\special{dt 0.045}%
}}%
%
{\color[named]{Black}{%
\special{pn 4}%
\special{sh 1}%
\special{ar 1600 400 14 14 0  6.28318530717959E+0000}%
\special{sh 1}%
\special{ar 1300 2200 14 14 0  6.28318530717959E+0000}%
\special{sh 1}%
\special{ar 2400 2200 14 14 0  6.28318530717959E+0000}%
\special{sh 1}%
\special{ar 2400 1000 14 14 0  6.28318530717959E+0000}%
\special{sh 1}%
\special{ar 1850 1000 14 14 0  6.28318530717959E+0000}%
\special{sh 1}%
\special{ar 1850 1000 14 14 0  6.28318530717959E+0000}%
}}%
%
{\color[named]{Black}{%
\special{pn 13}%
\special{pa 1876 1000}%
\special{pa 2376 1000}%
\special{fp}%
\special{sh 1}%
\special{pa 2376 1000}%
\special{pa 2308 980}%
\special{pa 2322 1000}%
\special{pa 2308 1020}%
\special{pa 2376 1000}%
\special{fp}%
\special{pa 2400 1026}%
\special{pa 2400 2176}%
\special{fp}%
\special{sh 1}%
\special{pa 2400 2176}%
\special{pa 2420 2108}%
\special{pa 2400 2122}%
\special{pa 2380 2108}%
\special{pa 2400 2176}%
\special{fp}%
\special{pa 2376 2200}%
\special{pa 1326 2200}%
\special{fp}%
\special{sh 1}%
\special{pa 1326 2200}%
\special{pa 1392 2220}%
\special{pa 1378 2200}%
\special{pa 1392 2180}%
\special{pa 1326 2200}%
\special{fp}%
}}%
\put(0.8500,-13.1000){\makebox(0,0)[lt]{$\tfrac{3}{2}$}}%
\end{picture}%
\end{center}
\caption{Choice of parameters in the proof of norm inflation, when $s<1$ (left) and $1\le s<\frac{3}{2}$ (right).}
\end{figure}
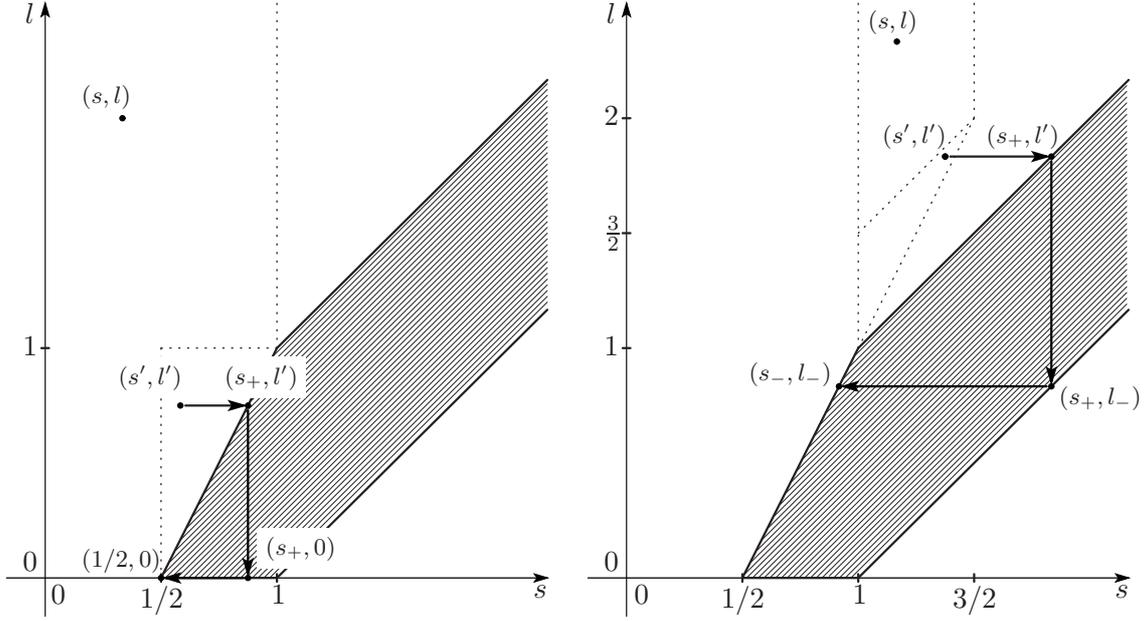
We choose initial data as $u_{0,N}:=N^{-s'}f _N$.
Then, 
\[ \norm{u_{0,N}}{H^{s'}}\sim 1,\qquad \norm{u_{0,N}}{H^s}\sim N^{s-s'}=o(1) \quad (N\to \I ),\]
and the calculation in the proof of Lemma~\ref{lem:notc2}~(ii) shows that
\eq{quad_dominant1}{\norm{n^{(2)}[u_{0,N}](t)}{H^{l'}}\gec |t|N^{l'-2s'+1}}
for any $t\neq 0$ sufficiently close to zero and any $N\gg |t|^{-1}$.

Since $\tnorm{u_{0,N}}{H^{s'}}\sim 1$, we can easily obtain a solution $(u_N,n_N)$ to the initial value problem with data $(u_{0,N},0,0)$, in the space $X^{s',\frac{1}{2},1}_S(T)\times C([-T,T];L^2)
$ with some $T>0$ independent of $N$, by solving the integral equation for $u_N$
\eq{IE-u}{u_N(t)=e^{it\Delta}u_{0,N}+i\la \int _0^te^{i(t-t')\Delta}\big[ u(t')\int _0^{t'}\frac{\sin ((t'-t'')|\nabla |)}{|\nabla |}\Delta _\be \big[ u_N(t'')\bbar{u_N(t'')}\big] \,dt''\big] \,dt'}
with the aid of linear and bilinear estimates \eqref{est_lin_sol}, \eqref{est_be_s}, \eqref{est_be_w} at the regularity $(s',0)$, and then define $n_N$ by the formula
\eqq{n_N(t)&=-\int _0^t\frac{\sin ((t-t')|\nabla |)}{|\nabla |}\Delta _\be \big[ u_N(t')\bbar{u_N(t')}\big] \,dt'.}
Moreover, we recall Remark~\ref{rem_wavemoissho2} and apply to \eqref{IE-u} the estimates \eqref{est_lin_sol}, \eqref{est_be_s} with $(\sgm ,0)$, $\sgm =s_+$ or $\frac{1}{2}$, and \eqref{est_be_w} with $(s',0)$, to obtain
\eqq{\norm{u_N}{X^{\sgm ,\frac{1}{2},1}_S(T)}\lec \norm{u_{0,N}}{H^{\sgm}}+T^{1-}\norm{u_N}{X^{\sgm ,\frac{1}{2},1}_S(T)}\norm{u_N}{X^{s',\frac{1}{2},1}_S(T)}^2.}
Since $\tnorm{u_N}{X^{s',\frac{1}{2},1}_S(T)}\lec \tnorm{u_{0,N}}{H^{s'}}\sim 1$, we have
\eqq{\norm{u_N}{X^{\sgm ,\frac{1}{2},1}_S(T)}\lec \norm{u_{0,N}}{H^{\sgm}}\sim N^{\sgm -s'}}
with $\sgm =s_+$ or $\frac{1}{2}$ and $T$ sufficiently small (still independent of $N$).
From a similar argument, we also conclude that the Duhamel term in \eqref{IE-u} is much smaller in $H^{s_+}$ than $u_N$, namely
\eqq{\norm{u_N-e^{it\Delta}u_{0,N}}{X^{s_+,\frac{1}{2},1}_S(T)}\lec T^{1-}\norm{u_N}{X^{s_+,\frac{1}{2},1}_S(T)}\norm{u_N}{X^{\frac{1}{2},\frac{1}{2},1}_S(T)}^2\lec T^{1-}N^{s_++1-3s'}.}

We use the above estimates to measure the difference between $n^{(2)}[u_{0,N}]$ and $n_N$ in $H^{l'}$.
Notice that
\eqq{n^{(2)}[u_{0,N}]-n_N&=\int _0^t\frac{\sin ((t-t')|\nabla |)}{|\nabla |}\Delta _\be F_N(t')\,dt',\\
F_N(t')&=u_N(t')\bbar{u_N(t')}-e^{it'\Delta}u_{0,N}\bbar{e^{it'\Delta}u_{0,N}}\\
&=u_N(t')\big( \bbar{u_N(t')-e^{it'\Delta}u_{0,N}}\big) +\big( u_N(t')-e^{it'\Delta}u_{0,N}\big) \bbar{e^{it'\Delta}u_{0,N}}.}
Hence, we employ \eqref{est_be_w} with regularity $(s_+,l')$ to obtain
\eqq{&\sup _{|t|\le T}\norm{n^{(2)}[u_{0,N}](t)-n_N(t)}{H^{l'}}\\
&\lec \norm{\int _0^t\frac{e^{-i(t-t')|\nabla |}}{|\nabla |}\Delta _\be F_N(t')\,dt'}{X^{l',\frac{1}{2},1}_{W_+}(T)}+\norm{\int _0^t\frac{e^{i(t-t')|\nabla |}}{|\nabla |}\Delta _\be F_N(t')\,dt'}{X^{l',\frac{1}{2},1}_{W_-}(T)}\\
&\lec T^{\frac{1}{2}-}\Big( \norm{u_N}{X^{s_+,\frac{1}{2},1}_S(T)}+\norm{e^{it\Delta}u_{0,N}}{X^{s_+,\frac{1}{2},1}_S(T)}\Big) \norm{u_N-e^{it\Delta}u_{0,N}}{X^{s_+,\frac{1}{2},1}_S(T)}\\
&\lec T^{\frac{3}{2}-}N^{2s_++1-4s'}=T^{\frac{3}{2}-}N^{l'-4s'+2}.}
Combining this with \eqref{quad_dominant1} and the fact that $2s'-1>0$, we have
\eqq{\norm{n_N(t)}{H^{l}}\ge \norm{n_N(t)}{H^{l'}}\ge c|t|N^{l'-2s'+1}-c'|t|^{\frac{3}{2}-}N^{l'-4s'+2}\gec |t|N^{l'-2s'+1}}
for $0<|t|\ll 1$ and sufficiently large $N$, which shows (i) for the case $s<1$.

For the proof of (ii), we take $s'=s_+=\frac{1}{2}$, $l'=0$ and repeat the above argument.
Sufficiently small $|t|$ then allows us to obtain
\eqq{\norm{n_N(t)}{L^2}\ge c|t|-c'|t|^{\frac{3}{2}-}\gec |t|,}
while letting $N\to \I$ shrinks the initial data in $H^s$, obtaining (ii).

The proof of norm inflation for the case $1\le s<\frac{3}{2}$ is parallel to the case $s<1$, so we will omit the details.
For $(s,l)$ satisfying $1\le s<\frac{3}{2}$ and $l>2s-1$, we choose $(s',l')$ such that $s<s'$, $l\ge l'$, $1<2s'-1<l'<s'+\frac{1}{2}$, and then take $s_+,l_-,s_-$ so that $l'=s_+$, $l_-=s_+-1$, and $l_-=2s_--1$ (see Figure~3, the right one).
For the same initial data $u_{0,N}$, we can show that
\eqq{\norm{n^{(2)}[u_{0,N}](t)}{H^l}&\gec |t|N^{l'-2s'+1},\qquad 0<|t|\ll 1,\quad N\gg |t|^{-1},\\
\norm{u_N}{X^{\sgm ,\frac{1}{2},1}_S(T)}&\lec \norm{u_{0,N}}{H^{\sgm}}\sim N^{\sgm -s'},\qquad \sgm =s_-,s',s_+,\\
\norm{u_N-e^{it\Delta}u_{0,N}}{X^{s_+,\frac{1}{2},1}_S(T)}&\lec N^{s_++2s_--3s'},\\
\sup _{|t|\le T}\norm{n^{(2)}[u_{0,N}](t)-n_N(t)}{H^{l'}}&\lec N^{2s_++2s_--4s'}=N^{3l'-4s'}.}
Since $l'<s'+\frac{1}{2}$ implies $l'-2s'+1>3l'-4s'$, we conclude the norm inflation (i).

At the end, we give a proof of (iv) to conclude this section.

Recall the definition of $K_N$, $\ti{K}_N$, and \eqref{assump:res}.
Put
\eqq{\ti{u}(\tau ,k)&:=\de _{K_N}(k)\chf{[-10\ga _2^{-2},10\ga _2^{-2}]}(\tau +|K_N|^2),\\
\ti{v}(\tau ,k)&:=\de _{\ti{K}_N}(k)\chf{[-10\ga _2^{-2},10\ga _2^{-2}]}(\tau +|\ti{K}_N|^2),\\
\ti{w}(\tau ,k)&:=\de _{\ti{K}_N-K_N}(k)\chf{[-10\ga _2^{-2},10\ga _2^{-2}]}(\tau +|K_N-\ti{K}_N|)}
for large $N\in \Bo{N}$.
Note that
\eqs{\norm{u}{X^{s,b,p}_S}\sim \norm{v}{X^{s,b,p}_S}\sim N^s,\quad \norm{\F _{\tau ,k}^{-1}\ti{w}(\tau ,\pm k)}{X^{l,b,p}_{W_+}}=\norm{\F _{\tau ,k}^{-1}\ti{w}(-\tau ,\pm k)}{X^{l,b,p}_{W_-}}\sim N^l}
for any $b$ and $p$.
We easily verify that
\eqq{\ti{uw}(\tau ,k)\gec \ti{v}(\tau ,k),\quad \ti{v\bar{w}}(\tau ,k)\gec \ti{u}(\tau ,k),\quad \ti{u\bar{v}}(\tau ,k)\gec \ti{w}(-\tau ,k),\quad \ti{\bar{u}v}(\tau ,k)\gec \ti{w}(\tau ,-k),}
which imply the estimates
\eqs{\norm{uw}{X^{s,b-1,p}_S}\gec N^s,\qquad \norm{u}{X^{s,b,p}_S}\norm{w}{X^{l,b,p}_{W_+}}\sim N^sN^l,\\
\norm{v\bar{w}}{X^{s,b-1,p}_S}\gec N^s,\qquad \norm{v}{X^{s,b,p}_S}\norm{\bar{w}}{X^{l,b,p}_{W_-}}\sim N^sN^l,\\
\norm{\frac{\Delta _\be}{\LR{\nabla}}(u\bar{v})}{X^{l,b-1,p}_{W_-}}\gec N^{l+1},\qquad \norm{u}{X^{s,b,p}_S}\norm{v}{X^{s,b,p}_S}\sim N^{2s},\\
\norm{\frac{\Delta _\be}{\LR{\nabla}}(v\bar{u})}{X^{l,b-1,p}_{W_+}}\gec N^{l+1},\qquad \norm{v}{X^{s,b,p}_S}\norm{u}{X^{s,b,p}_S}\sim N^{2s}.}
We can easily disprove \eqref{behanrei1} for $l<0$ and \eqref{behanrei2} for $l>2s-1$ by using these estimates.

The other cases, namely \eqref{behanrei1} for $l<s-1$ and \eqref{behanrei2} for $l>s$, have been already treated in \cite{Ta} in the 1d setting, and the proof in \cite{Ta} can be applied to our case $d\ge 2$ with some trivial modification.
We will omit the details.


\bigskip
\section*{Acknowledgments}

\noindent The author would like to express his great appreciation to Professor Yoshio Tsutsumi for a number of precious suggestions.
This work was partially supported by Grant-in-Aid for JSPS Fellows 08J02196.


\end{document}